\documentclass[10pt]{article}
\usepackage{geometry}                % See geometry.pdf to learn the layout options. There are lots.
\usepackage{graphicx}
\usepackage{fullpage}
\usepackage{amsmath}
\usepackage{amssymb}
\usepackage{amsthm}
\usepackage{url}
\usepackage{epsfig}
\usepackage{color}
\usepackage{refcount}
\usepackage{verbatim}
\usepackage{enumerate}
\usepackage{multirow}
\usepackage{framed}
\usepackage{slashbox}
\usepackage{caption}
\usepackage{subcaption}
\usepackage[normalem]{ulem}

\usepackage[square,numbers,sort&compress]{natbib}

\usepackage{xr}
\usepackage[breaklinks=true,colorlinks,citecolor=blue,linkcolor=blue]{hyperref}

\usepackage{calc}

\usepackage{tikz}
\usetikzlibrary{calc}
\usepackage{epstopdf}
\usetikzlibrary{positioning}
\usepackage{pgffor}

\usepackage{color}
\definecolor{clemson-orange}{RGB}{234,106,32}
\definecolor{chicago-maroon}{RGB}{128,0,0}
\definecolor{cincinnati-red}{RGB}{190,0,0}
\definecolor{soft-cyan}{RGB}{68,85,90}
\usepackage{fullpage}
\usepackage{multicol}

\usepackage[utf8]{inputenc}

\usepackage{array}
\newcolumntype{L}[1]{>{\raggedright\let\newline\\\arraybackslash\hspace{0pt}}m{#1}}
\newcolumntype{C}[1]{>{\centering\let\newline\\\arraybackslash\hspace{0pt}}m{#1}}
\newcolumntype{R}[1]{>{\raggedleft\let\newline\\\arraybackslash\hspace{0pt}}m{#1}}

\newcommand{\bb}{\mathbb}
\newcommand{\R}{\bb R}
\newcommand{\Z}{{\bb Z}}
\newcommand{\N}{{\bb N}}

\theoremstyle{definition}
\newtheorem{theorem}{Theorem}[section]
\newtheorem{lemma}[theorem]{Lemma}
\newtheorem{corollary}[theorem]{Corollary}

\newtheorem{claim}{Claim}
\newtheorem{definition}[theorem]{Definition}
\newtheorem{remark}[theorem]{Remark}

\newtheorem{example}[theorem]{Example}
\newtheorem*{question*}{Question}

\DeclareMathOperator*{\conv}{conv}

\DeclareMathOperator*{\poly}{poly}

\newcommand{\D}{\mathcal{D}}

\newcommand{\cL}{\mathcal{L}}

\numberwithin{equation}{section}  % If you number theorems, etc. within sections,
\title{Complexity of branch-and-bound and cutting planes in mixed-integer optimization}
\author{Amitabh Basu\thanks{Department of Applied Mathematics and Statistics, Johns Hopkins University, Baltimore, MD, USA ({\tt basu.amitabh@jhu.edu}, {\tt hjiang32@jhu.edu}).} 
\and Michele Conforti\thanks{Dipartimento di Matematica ``Tullio Levi-Civita'', Universit\`a degli Studi Padova, Italy ({\tt conforti@math.unipd.it}, {\tt disumma@math.unipd.it}).}
\and Marco Di Summa\footnotemark[3]
\and Hongyi Jiang\footnotemark[2]}
\begin{document}

\maketitle

\begin{abstract}We investigate the theoretical complexity of branch-and-bound (BB) and cutting plane (CP) algorithms for mixed-integer optimization. In particular, we study the relative efficiency of BB and CP, when both are based on the same family of disjunctions. We extend a result of Dash to the nonlinear setting which shows that for convex 0/1 problems, CP does at least as well as BB, with variable disjunctions. We sharpen this by giving instances of the stable set problem where we can provably establish that CP does exponentially better than BB. We further show that if one moves away from 0/1 sets, this advantage of CP over BB disappears; there are examples where BB finishes in $O(1)$ time, but CP takes infinitely long to prove optimality, and exponentially long to get to arbitrarily close to the optimal value (for variable disjunctions). We next show that if the dimension is considered a fixed constant, then the situation reverses and BB does at least as well as CP (up to a polynomial blow up), no matter which family of disjunctions is used. This is also complemented by examples where this gap is exponential (in the size of the input data).\end{abstract}

\section{Introduction}\label{sec:intro}

In this paper, we consider the following optimization problem:

\begin{equation}\label{eq:linear-obj}
\begin{array}{rcll}
\sup\limits_{x \in \R^n} & \langle c, x \rangle && \\
\textrm{s.t.} & x &\in & C \\
& x &\in & S
\end{array}
\end{equation}
where $C$ is a closed, convex subset of $\R^n$ and $S$ is a closed, possibly non-convex, subset of $\R^n$. This model is a formal way to ``decompose" the feasible region into the ``convex" constraints $C$ and the ``non-convexities" $S$ of the problem at hand. The bulk of this paper will be concerned with non-convexity coming from integrality constraints, i.e., $S := \Z^{n_1} \times \R^{n_2}$, where $n_1 + n_2 = n$; the special case $n_2 = 0$ will be referred to as a {\em pure-integer lattice} and the general case as a {\em mixed-integer lattice} ($n_1=0$ gives us standard continuous convex optimization). However, some of the ideas put forward apply to other non-convexities like sparsity or complementarity constraints as well (see Theorem \ref{thm:BB<=CP} below, where the only assumption on $S$ is closedness).

\paragraph{Cutting Planes and Branch-and-Bound.}
Cutting planes were first successfully employed for solving combinatorial problems with special structure, such as the traveling salesman problem~\cite{applegate2011traveling,grotschel1979symmetric, grotschel1979symmetric-2,grotschel1985polyhedral,grotschel1986clique, chvatal1973edmonds-II, dantzig1954solution,dantzig1959linear, cornuejols1982travelling, cornuejols1985travelling, cornuejols1983halin}, the independent set problem~\cite{nemhauser1974properties,padberg1973facial,chvatal1975certain,trotter1975class}, the knapsack problem~\cite{balas1975facets,wolsey1975faces}, amongst others. For general mixed-integer problems, cutting plane ideas were introduced by Gomory~\cite{gomory1960algorithm,MR0102437}, but did not make any practical impact until the mid 1990s~\cite{balas96gomory}. Since then, cutting planes have been cited as the most significant component of modern solvers~\cite{Bixby}, where they are combined with a systematic enumeration scheme called {\em Branch-and-Bound}. Both of these ideas are based on the following notion.
\medskip

\begin{definition}\label{def:disj} Given a closed subset $S \subseteq \R^n$, a {\em disjunction covering $S$} is a finite union of closed convex sets $D = Q_1 \cup \ldots \cup Q_k$ such that $S \subseteq D$. Such a union is also called a {\em valid} disjunction.
\end{definition}
\medskip

Observe that the feasible region of~\eqref{eq:linear-obj} is always contained in any valid disjunction $D$. This leads to a fundamental algorithmic idea: one iteratively refines the initial convex ``relaxation" $C$ by intersecting it with valid disjunctions. More formally, a {\em cutting plane for $C$ derived from a disjunction} $D$ is any halfspace $H\subseteq \R^n$ such that $C \cap D \subseteq H$. The point is that the feasible region $C\cap S \subseteq C\cap D \subseteq C \cap H$.

Thus, the convex region $C$ is refined or updated to a tighter convex set $C \cap H$. The hope is that iterating this process with clever choices of disjunctions and cutting planes derived from them will converge to the convex hull of $C \cap S$, where the problem can be solved with standard convex optimization tools. Since the objective is linear, solving over the convex hull suffices.
\bigskip

\begin{figure}[htbp]
\begin{multicols}{2}
\begin{center}\includegraphics[width=0.3\textwidth]{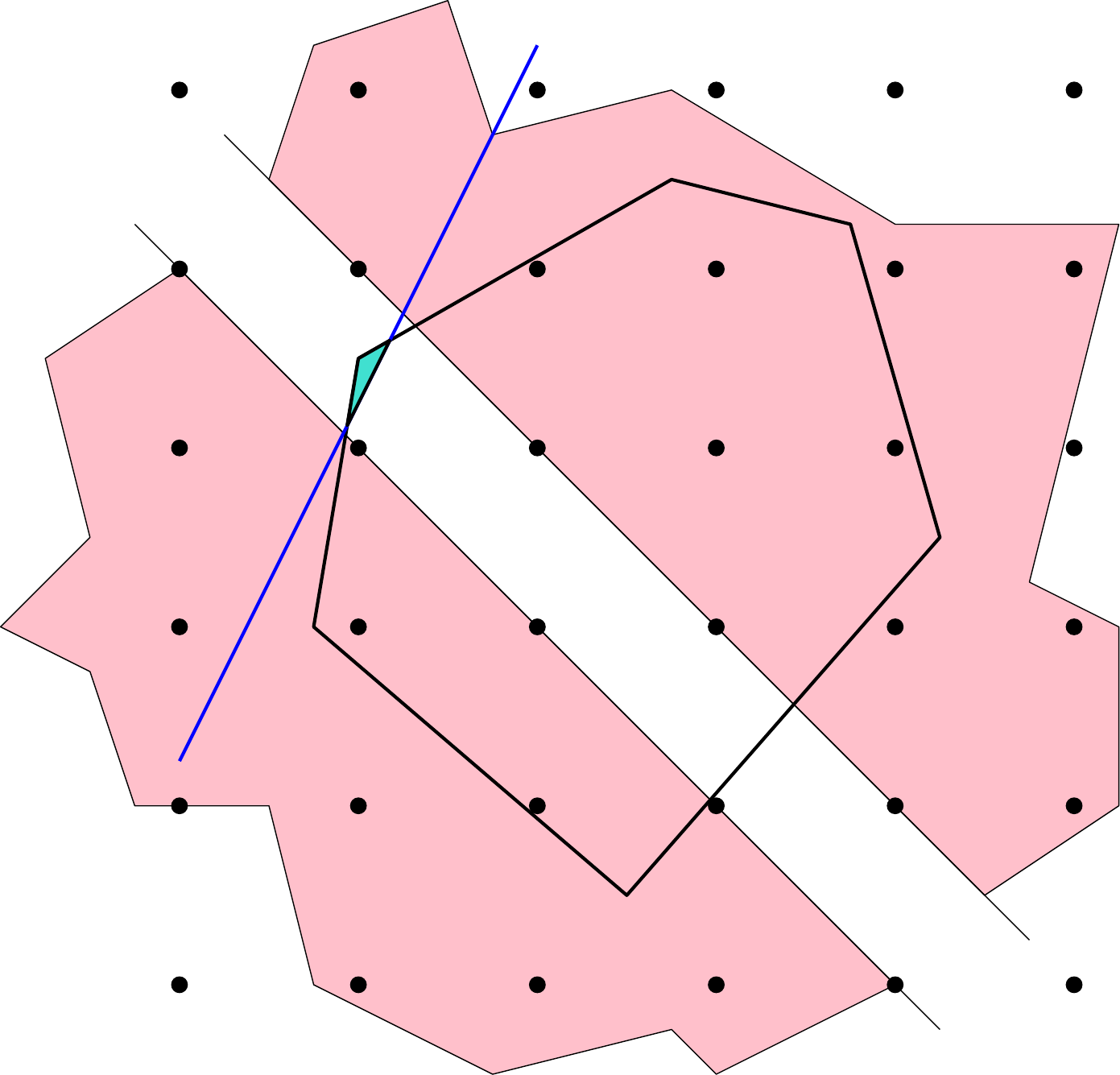}\end{center}
\columnbreak
\begin{center}\includegraphics[width=0.27\textwidth]{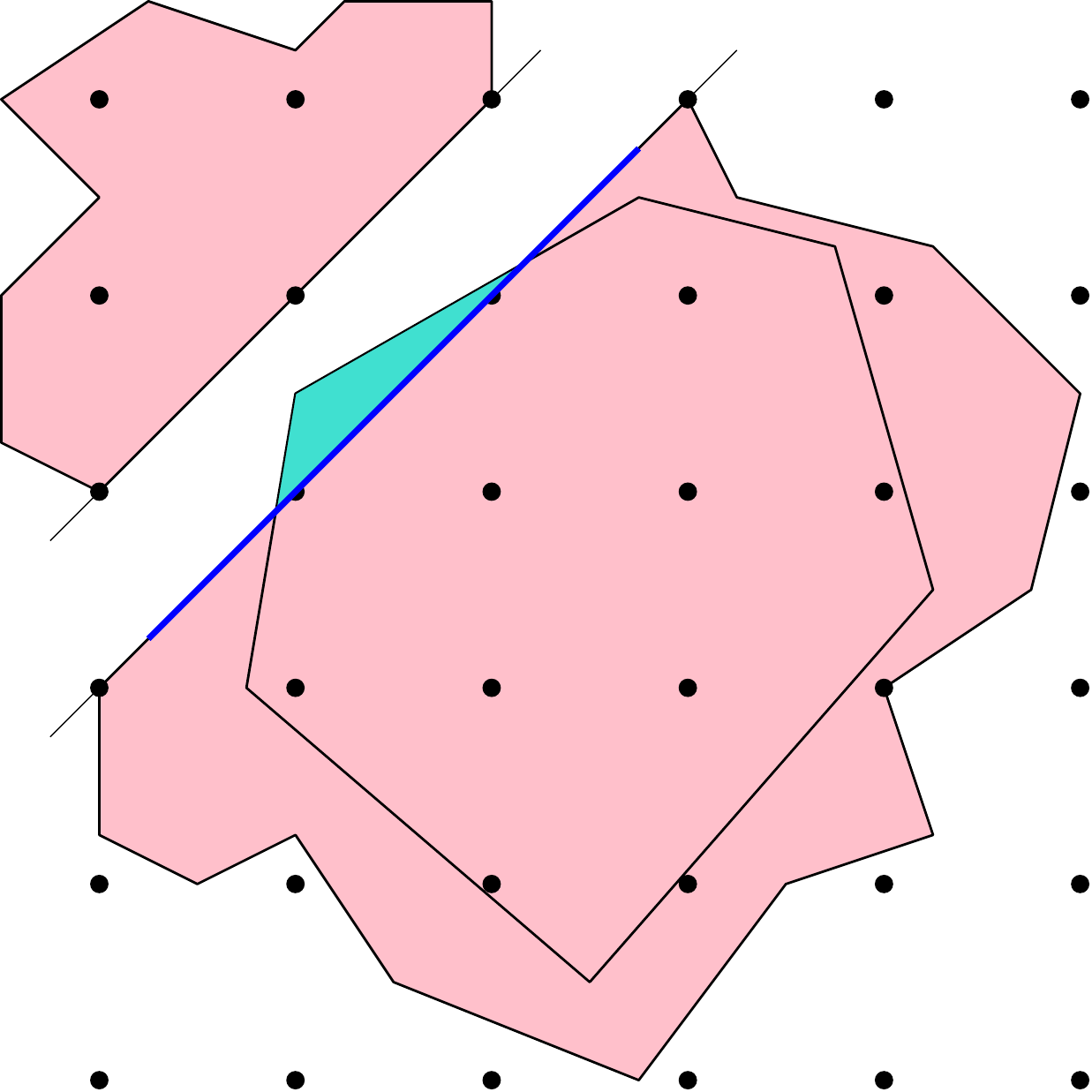}\end{center}
\end{multicols}
\caption{Two examples of cutting planes based on split disjunctions for $S=\Z^n$. The convex region $C$ is a polytope. The dashed line shows the bounding hyperplane of the cutting plane $H$; the dark triangle in both cases is the part of $C$ that is ``shaved off", i.e. $C \setminus H$. }\label{fig:splits}
\end{figure}

\begin{example}\label{ex:splits}
{\em Split disjunctions} for the mixed-integer lattice were introduced by Cook, Kannan and Schrijver~\cite{cook-kannan-schrijver}. These are disjunctions that are a union of two rational halfspaces that cover the mixed-integer lattice. Note that this implies the bounding hyperplanes of the two halfspaces have to be parallel. See Figure~\ref{fig:splits} where the disjunctions are colored in light gray. The right figure illustrates an example of {\em Chv\'atal-Gomory cuts}~\cite{chvatal1973edmonds}, which are cutting planes derived from split disjunctions where one side of the disjunction does not intersect the convex set $C$. 
Most cutting planes used in combinatorial optimization are Chv\'atal-Gomory cuts. For example, in the Maximum Matching problem the so-called {\em odd set} or {\em blossom inequalities} are an example of Chv\'atal-Gomory cuts~\cite{edmonds1965maximum,edmonds1965paths,chvatal1973edmonds}, where $C$ is the standard polyhedral formulation for maximum matching with 0/1 variables for the edges of the graph.
\end{example} 

We will now formally define algorithms based on cutting planes and branch-and-bound, assuming access to a continuous, convex optimization solver. Below, when we use the word ``solve" to process a continuous convex optimization problem, we assume that the output of such a solver will either  (i) report infeasibility, or (ii) report that a maximizer does not exist either because the problem is unbounded or because the supremum is not attained, or (iii) report an optimal solution to the convex optimization problem.

\begin{definition}\label{def:CP-algo} Cutting plane (CP) algorithm based on a family $\mathcal{D}$ of disjunctions:
\begin{enumerate}
\item Initialize $C_0 := C$.
\item For $i = 0, 1, \ldots$
\begin{enumerate}
\item[(a)] Solve $x^i \in \arg\max\{ \langle c, x\rangle: x \in C_i\}$. If $C_i = \emptyset$, report ``INFEASIBLE" and STOP. If no maximizer exists, report ``EXCEPTION" and STOP.
\item[(b)] If $x^i \in S$, report $x^i$ as OPTIMAL and STOP. Else, {\bf choose} a disjunction $D \in \mathcal{D}$ and a cutting plane $H$ for $C_i$ derived from $D$ such that $x^i \not\in H$. Set $C_{i+1} := C_i \cap H$. If no cutting plane can be derived, report ``NO CUTTING PLANE" and STOP.
\end{enumerate}
\end{enumerate}
\end{definition}
The outputs ``EXCEPTION'' and ``NO CUTTING PLANE'' correspond to situations in which the CP algorithm stops without finding an optimal solution to the given problem. Note however that if, e.g., $C$ is compact, then the output ``EXCEPTION'' will never occur. Also, if $S$ is the mixed-integer lattice and $\mathcal D$ is the family of split disjunctions, the output ``NO CUTTING PLANE'' will never occur if $x^i$ is an extreme point of $C_i$. We call the sequence of operations in Definition~\ref{def:CP-algo} an ``algorithm", even though it may not terminate in finitely many iterations.

In the framework above, at every iteration there are usually many possibilities for the choice of the disjunction from $\mathcal{D}$, and then many possibilities for the choice of the cutting plane from the chosen disjunction. Specific strategies for these two choices give a particular instance of a cutting plane algorithm based on the family of disjunctions $\mathcal{D}$. 

Disjunctions can also be used to simply search, as opposed to convexification by cutting planes. This leads to the idea of {\em branching} with pruning by bounds.
\medskip

\begin{definition}\label{def:BB-algo} Branch-and-bound (BB) algorithm based on a family $\mathcal{D}$ of disjunctions: \begin{enumerate}
\item Initialize a list $\mathcal{L} := \{C\}$, $LB := -\infty$.
\item While $\mathcal{L} \neq \emptyset$
\begin{enumerate}
\item[(a)] {\bf Choose} $N \in \mathcal{L}$ and update $\mathcal{L} := \mathcal{L}\setminus \{N\}$. Solve $\bar x \in \arg\max\{ \langle c, x\rangle: x \in N\}$. If $N= \emptyset$, continue the loop. If no maximizer exists, report ``EXCEPTION" and STOP.
\item[(b)] If $\langle c,\bar x\rangle > LB$, then check if $\bar x \in S$. If yes, update $LB := \langle c,\bar x\rangle$; if no, {\bf choose} a disjunction $D = (Q_1 \cup \ldots \cup Q_k) \in \mathcal{D}$ such that $\bar x \not\in D\cap N$ and update $\mathcal{L} := \mathcal{L} \cup \bigcup_{i=1}^k(Q_i\cap N)$. If no such disjunction exists, report ``NO DISJUNCTION FOUND" and STOP.
\end{enumerate}
\item If $LB=-\infty$ report ``INFEASIBLE". Else, return $\bar x \in S$ corresponding to current $LB$ as OPTIMAL and STOP.
\end{enumerate}
\end{definition}

The idea is to maintain a list $\mathcal{L}$ of {\em convex subsets} of the initial convex set $C$ which are guaranteed to contain the optimal point. $LB$ stores the objective value of the best feasible solution found so far, which is a lower bound for the optimal value. In the worst case, one will go through each integer point (or connected component of $S$) in $C$.

Cutting planes and branch-and-bound can be combined into a single algorithm as well that generalizes Definitions~\ref{def:CP-algo} and~\ref{def:BB-algo}.

\begin{definition} A {\em branch-and-cut (BC)} algorithm is a version of the algorithm defined in Definition~\ref{def:BB-algo}, where there is an additional decision point in Step 2 (b), where one decides if one should add a cutting plane or branch as described in the step.
\end{definition}

The literature on the complexity of cutting plane methods has often focused on the concepts of {\em cutting plane proof} and {\em cutting plane rank}, which are closely related to the efficiency of a cutting plane algorithm (Definition~\ref{def:CP-algo}).
\medskip

\begin{definition}\label{def:CP-proof-rank} Let $C\subseteq \R^n$ be a convex set, let $S\subseteq \R^n$ model the non-convexity, and let $\mathcal{D}$ be a family of valid disjunctions for $S$ (see Definition~\ref{def:disj}).
\begin{enumerate}
\item Suppose the inequality $\langle c, x \rangle \leq \gamma$ is valid for all points in $C \cap S.$ A {\em cutting plane (CP) proof based on $\D$ with respect to $C\cap S$} of length $N \in \N$ for this inequality is a sequence of halfspaces $H_1, H_2, \ldots, H_N$ such that 1) there exists a sequence of disjunctions $D_1, D_2, \ldots, D_N \in \mathcal{D}$ such that for each $i=0, \ldots, N-1$, $H_{i+1}$ is a cutting plane derived from the disjunction $D_{i+1}$ applied to $C_i := C \cap \bigcap_{j=1}^{i} H_j,$\footnote{We make the standard notational convention that the trivial intersection $\bigcap_{i=1}^0 X_i = \R^n$.} and 2) $\langle c, x \rangle \leq \gamma$ is valid for all points in $C_N$.
\item The {\em $\mathcal{D}$-closure} is defined as the intersection of $C$ with all cutting planes that can be derived from all disjunctions in $\mathcal{D}$, which we denote by $CP_\D(C)$. We can iterate this operator $N\geq 1$ times, which we will denote by $CP^N_\D(C) = \underbrace{CP_\D(CP_\D( \ldots CP_\D}_{N \textrm{times}}(C) \ldots ))$. The {\em $\mathcal{D}$-rank} of a valid inequality $\langle c, x \rangle \leq \gamma$ for $C \cap S$ is the smallest $N \in \N$ such that this inequality is valid for $CP^N_\D(C)$ (the rank is $\infty$ if the inequality is invalid for all $CP^N_\D(C)$, $N\geq 1$). \end{enumerate}
\end{definition}

The relation to the notion of a cutting plane algorithm (Definition~\ref{def:CP-algo}) is simple: Suppose an instance of~\eqref{eq:linear-obj} has optimal value $OPT$ and a cutting plane algorithm ends with this value in $N$ iterations. Clearly, the cutting planes generated during the algorithm give a cutting plane {\em proof} of length $N$ for the inequality $\langle c, x \rangle \leq OPT$. Thus, establishing lower bounds on the length of cutting plane proofs are a way to derive lower bounds on the efficiency of a cutting plane algorithm. Also, proving a lower bound on the {\em rank} of an inequality gives a lower bound on the cutting plane proof length. See~\cite{dash2002exponential,dash2005exponential,dash2010complexity,chvatal1989cutting,chvatal1984cutting,chvatal1980hard,cook2001matrix,bockmayr1999chvatal,eisenbrand2003bounds,rothvoss20130} for a sample of upper and lower bound results on rank and proof length. 

One can define analogous concepts for branch-and-bound. In fact, we will now define a generalization of the notion of cutting plane proof to branch-and-cut procedures.
\medskip

\begin{definition}\label{def:BC-proof-rank} Let $C\subseteq \R^n$ be a convex set, let $S\subseteq \R^n$ model the non-convexity, and let $\mathcal{D}$ be a family of valid disjunctions for $S$ (see Definition~\ref{def:disj}). Suppose the inequality $\langle c, x \rangle \leq \gamma$ is valid for all points in $C \cap S.$ A {\em branch-and-cut (BC) proof based on $\D$ with respect to $C\cap S$} for this inequality is a rooted tree such that 1) every node represents a convex subset of $C$, 2) the root node represents $C$ itself, 3) every non-leaf node is labeled as a {\em cutting} node or a {\em branching} node, 4) any cutting node representing $C' \subseteq C$ has exactly one child and there exists a disjunction $D \in \D$ and a cutting plane $H$ for $C'$ derived from $D$ such that the child represents $C' \cap H$, 5) for any branching node representing $C' \subseteq C$, there exists a disjunction in $\D$ given by $D = Q_1 \cup \ldots\cup Q_k$ such that the children of this node represent the sets $C' \cap Q_i$, $i=1, \ldots, k$, and 6) $\langle c, x \rangle \leq \gamma$ is valid for all the subsets represented by the leaves of the tree. The {\em size} of the BC proof is the total number of nodes in the tree minus 1 (we exclude the root representing $C$). \end{definition}

\begin{remark} Note that if all nodes in a branch-and-cut proof are cutting nodes, then we have a cutting plane proof and notions of length and size coincide. If all nodes are branching nodes, then we say we have a {\em branch-and-bound (BB) proof}.
\end{remark}

As with the case of cutting planes, a branch-and-bound/branch-and-cut algorithm also provides a branch-and-bound/branch-and-cut proof of optimality with size equal to the number of iterations of the algorithm. Also, note that by just changing the stopping criterion in the algorithms, one can use them to derive upper bounds on the objective function value, i.e., proving validity of $\langle c, x \rangle \leq \gamma$, instead of stopping at optimality or infeasibility. Consequently, CP/BB/BC proofs can be seen as a way to prove unsatisfiability in logic. This led to a rich literature at the intersection of optimization and proof complexity~\cite{beame_et_al:LIPIcs:2018:8341,dadush2020complexity,bonet1997lower,razborov2017width,impagliazzo1994upper,buss1996cutting,cook1987complexity,goerdt1990cutting,goerdt1991cutting,clote1992cutting,pudlak1997lower,pudlak1999complexity,krajivcek1998discretely,kojevnikov2007improved,grigoriev2002complexity}, to cite a few.

The difference between a CP/BB/BC algorithm and a CP/BB/BC proof of optimality (or validity of an upper bound) is that the algorithm is restricted to choose a disjunction (and a cutting plane based on this disjunction) that ``shaves off" the current optimal solution. There is no such restriction on a proof. In fact, there are two-dimensional examples where every CP algorithm based on a standard family of disjunctions converges to a value that is 12.5\% larger than the optimal value even after infinitely many iterations, while there are finite length CP proofs based on the same family that can certify the validity of any upper bound {\em strictly} bigger than the optimal value~\cite{owen2001disjunctive}.

\section{Main Results} We summarize our main results in Table~\ref{tab:CP-BB-results} and the discussion below.
\medskip

\begin{table}[htbp]
\begin{tabular}{|C{1.3cm}|L{3.3cm}|L{3.2cm}|L{3.3cm}|L{3.2cm}|}
\hline
& 0/1 convex sets, variable disjunctions & general convex sets, variable disjunctions & 0/1 convex sets, split disjunctions & general convex sets, split disjunctions \\
\hline
Variable dim. & CP $\leq$ BB {(Thm.~\ref{thm:CP<=BB})}& BB $O(1)$ vs. CP~$\infty$ {(Thm.~\ref{thm:BB<<CP})}& Open question & Open question \\
& & & & \\
 & CP $\poly(n)$ vs. BB~$\exp(n)$ {(Thm.~\ref{thm:CP<<BB})} & CP $\poly(n)$ vs. BB~$\exp(n)$  {(Thm.~\ref{thm:CP<<BB})}& &  \\
%  & &&&  \\
\hline
Fixed dim. & { BB and CP~$O(1)$ (Thm.~\ref{thm:fixed-dim-0/1})}& BB $\leq$ $\poly$(CP) {(Thm.~\ref{thm:BB<=CP})}& { BB and CP $O(1)$ (Thm.~\ref{thm:fixed-dim-0/1})} & BB $\leq$ $\poly$(CP) {(Thm.~\ref{thm:BB<=CP})}\\
& & & & \\
 &  & BB $O(1)$ vs. CP~$\exp(data)$ {(Thm.~\ref{thm:BB-con-CP-exp})}& & BB $O(1)$ vs. CP~$\poly(data)$ {(Remark~\ref{rem:CP-poly})}\\
\hline
\end{tabular}
\vspace{0.2cm}
\caption{The two rows distinguish between the cases when the dimension $n$ is considered variable or a fixed constant in the bounds reported. The columns consider combinations of different families of convex sets $C$ -- ``0/1 convex sets" or ``general convex sets" -- and two popular disjunction families in integer optimization. ``BB"/``CP" denotes the smallest size of a BB/CP proof or algorithm in the corresponding setting. }\label{tab:CP-BB-results}
\vspace{-0.5cm}
\end{table}
%\medskip

%\vspace{-0.3cm}
\paragraph{0/1 convex sets and variable disjunctions.} Consider the family of pure-integer instances where $S=\Z^n$ and $C$ is a convex set contained in the $[0,1]^n$ hypercube. This captures most combinatorial optimization problems, for instance. We focus on the most commonly used disjunctions in practice: \emph{variable disjunctions}, i.e., every disjunction is a union of two halfspaces of the form
\begin{equation}\label{eq:var-disj}
D_{i,K} := \{x \in \R^n: x_i \leq K\}\cup \{x \in \R^n: x_i \geq K+1\}, \qquad \textrm{where }i \in \{1, \ldots, n\} \textrm{ and }K \in \Z.
\end{equation}

The following is a generalization of a result by Dash~\cite{dash2002exponential,dash2005exponential}.

\begin{theorem}\label{thm:CP<=BB} Let $C\subseteq [0,1]^n$ be any closed, convex set. Let $S=\Z^n$ and $\D$ be the family of variable disjunctions. Let $\langle c, x \rangle \leq \gamma$ be a valid inequality for $C \cap \{0,1\}^n$ (possibly $c=0$, $\gamma = -1$ if $C \cap \{0,1\}^n = \emptyset$). If there exists a branch-and-cut proof of size $N$ based on $\D$ that certifies $\langle c, x \rangle \leq \gamma$, then for any $\epsilon > 0$, there exists a cutting plane proof based on $\D$ of size at most $N$ certifying $\langle c, x \rangle \leq \gamma + \epsilon$.

If $C$ is a polytope, then the statement is also true with $\epsilon = 0$.
\end{theorem}

Thus, in this setting, cutting planes are always at least as good as branch-and-bound (up to $\epsilon$ slack in the general convex case and exactly for the polyhedral case). The question arises if CP can be provably much better. We show that this is indeed the case with an instance of maximum independent set problems. Let $P(G)$ denote the standard formulation of the independent set problem on a graph $G = (V,E)$: $P(G):= \{x \in \R^V: x_u + x_v \leq 1 \;\; \forall (u,v) \in E\;\; x \in [0,1]^V\}$; so the independent sets are represented by $P(G) \cap \Z^n$. The objective is to maximize $\sum_{v\in V}x_v$.
\medskip

\begin{theorem}\label{thm:CP<<BB} Let $G$ be the graph given by $m$ disjoint copies of $K_3$ (cliques of size $3$). Then for $C=P(G)$  with objective $\sum_{v\in V}x_v$, $S= \Z^{3m}$ and $\D$ representing the family of all variable disjunctions, there is a cutting plane algorithm which solves the maximum independent set problem in $m$ iterations, but any branch-and-bound proof certifying an upper bound of $m$ on the optimal value has size at least $2^{m+1}-2$ for all $m\geq 1$.\end{theorem}
\medskip

Theorems~\ref{thm:CP<=BB} and~\ref{thm:CP<<BB} are together represented in the first row and first column of Table~\ref{tab:CP-BB-results}. Theorem~\ref{thm:CP<<BB} can be generalized to examples where the graph is connected;  see subsection~\ref{sec:misc}.

\begin{remark}
Examples where BB based on variable disjunctions takes an exponential number of iterations were already known: see, e.g., \cite{Jeroslow1974,Krishnamoorthy08}. However, to the best of our knowledge, it is unknown whether CP based on the same family of disjunctions can finish in polynomial time on these instances.
\end{remark}

\paragraph{General polytopes and variable disjunctions.} If one allows polytopes that are not necessarily subsets of the $n$-hypercube, but sticks with $S =\Z^n$ and variable disjunctions, then the advantage of CP over BB discussed above disappears. We created examples of polytopes $P_n$ in every dimension $n$ such that there is a BB algorithm that solves the problem in $O(1)$ iterations, but any CP proof will necessarily have {\em infinite length} (recall that only cutting planes based on variable disjunctions are allowed). 

 \begin{theorem}\label{thm:BB<<CP} Let $B \subseteq \R^2$ be defined as the convex hull of $\{(0,0), (1.5,1), (2,2), (1, 1.5)\}$. For every $n\in \N$, let $P_n\subseteq \R^n$ be the Cartesian product of $B$ and the hypercube $[-\frac12,\frac12]^{n-2}$, i.e., $P_n := B \times [-\frac12,\frac12]^{n-2}$. Consider instances of~\eqref{eq:linear-obj} with the objective $x_1 - x_2$, $C = P_n$ and $S=\Z^n$, and let $\D$ be the family of variable disjunctions. The optimal value is 0 and there is a branch-and-bound algorithm based on $\mathcal D$ which certifies $x_1 - x_2 \leq 0$ in $O(1)$ steps. However, any cutting plane proof of $x_1 - x_2 \leq 0$ has infinite length.
 \end{theorem}

Together with Theorem~\ref{thm:CP<<BB}, this fills in the first row and second column of Table~\ref{tab:CP-BB-results}.

\medskip

\paragraph{Fixing the dimension.} {The following relatively straightforward result shows that in fixed dimensions, one must go beyond the $0/1$ setting for anything interesting to happen.

\begin{theorem}\label{thm:fixed-dim-0/1} Let $C \subseteq [0,1]^n$ be compact, convex and let $\langle c, x \rangle \leq \gamma$ be valid for $C\cap \Z^n$. There exist BB and CP proofs based on variable disjunctions that prove the validity of this inequality and whose size is bounded by a function only of the dimension. Thus, if we consider instances in a fixed dimension, there are O(1) size BB and CP proofs.
\end{theorem}
It is actually easy to present a BB algorithm that takes $O(1)$ iterations (this is what we do in the proof of Theorem~\ref{thm:fixed-dim-0/1}); however, it is not clear if the CP proof presented in the proof can be converted into a CP algorithm that takes $O(1)$ iterations. This remains an open question.
}
\bigskip

The independent set example in Theorem~\ref{thm:CP<<BB} shows that if the dimension varies, CP can be exponentially better than BB. One can ask if a family of {polyhedral} instances can be constructed in some fixed dimension such that CP is better than BB by a factor that is exponential in the size of the input data (i.e., bit complexity of the numerical entries of the constraints). Interestingly, in fixed dimensions, the situation is {\em reversed}; roughly speaking BB is always as good as CP (at most polynomial blow-up), if the family of disjunctions has ``bounded complexity". Moreover, we can establish such a result for quite general non-convexities and convex relaxations.
\smallskip

\begin{definition} The {\em complexity} of a disjunction $D = Q_1 \cup \ldots \cup Q_k$ is defined as follows: if any of the $Q_i$'s are non-polyhedral, then the complexity of $D$ is $\infty$. Else, the complexity of $D$ is the sum of the number of facets of each $Q_i$. The {\em complexity} of a family $\D$ of disjunctions is the maximum complexity of any disjunction in $\D$ (e.g., the complexity of the family of split disjunctions is $2$).
\end{definition}
\medskip

\begin{theorem}\label{thm:BB<=CP} Fix $n \in \N$.  Let $S\subseteq \R^n$ be any closed set modeling the non-convexity and let $\D$ be a family of valid disjunctions for $S$ with complexity bounded by $M$. For any convex set $C\subseteq \R^n$ and any inequality $\langle c, x \rangle \leq \gamma$ valid for $C\cap S$, if there is a cutting plane proof of length {$K$} for its validity, then there is a branch-and-bound algorithm which proves the validity and takes at most {$O((MK)^{n+1})$} iterations.
\end{theorem}

A notion that becomes important in some branch-and-bound type algorithms is the so-called {\em flatness constant} of a convex set. The flatness theorem~\cite{BanaszczykLitvakPajorSzarek99,banaszczyk1996inequalities,rudelson2000distances} states that there exists a function $f :\N \to \N$, such that for every $n\in \N$, if $C\subseteq \R^n$ is a convex set with $C \cap \Z^n = \emptyset$, then there exists $w \in \Z^n\setminus\{0\}$ such that $\max_{x \in C}\langle w, x\rangle - \min_{x \in C}\langle w, x\rangle \leq f(n)$. The number $f(n)$ is usually referred to as the \emph{flatness constant} (where the word ``constant'' is justified by the fact that this function has its main use when the dimension $n$ is fixed). Combining the flatness theorem with the proof techniques that go behind Theorem~\ref{thm:BB<=CP}, we can also prove the following related theorem for the pure-integer lattice.

\begin{theorem}\label{thm:BB-algorithm-fixed-dim} Fix $n \in \N$.  Let $S=\Z^n$ and let $\D$ be the family of split disjunctions. For any convex set $C\subseteq \R^n$ and any inequality $\langle c, x \rangle \leq \gamma$ valid for $C\cap S$ with $c \in \Z^n$, there is a branch-and-bound algorithm which proves the validity of $\langle c, x \rangle \leq \gamma$ and takes at most $O((2f(n))^{n(n+1)})$ iterations, where $f(n)$ is the flatness constant.
\end{theorem}
\medskip

Theorems~\ref{thm:BB<=CP} and~\ref{thm:BB-algorithm-fixed-dim} provide some mathematical reasons for why the best complexity guarantees in fixed dimensions are for Lenstra-style algorithms, which can be interpreted as branch-and-bound algorithms. To complement this, the instances from Theorem~\ref{thm:BB<<CP} can be interpreted to be fixed dimension examples showing that BB can be infinitely better than CP. {Nevertheless, in that instance, there is an $O(\log(1/\epsilon))$ size CP proof for $\epsilon$-optimality, i.e., to prove $x_1 - x_2 \leq \epsilon$ for any $\epsilon > 0$. So the CP proof for approximate optimality is polynomial size in terms of the approximation parameter.} We present another family of instances in fixed dimensions in Theorem~\ref{thm:BB-con-CP-exp} below where there are CP algorithms that finish in finite time but any such algorithm will take exponentially (in the data size) many steps ({ we state the result for exact CP proofs, but the CP proofs remain exponential in size even when allowing for $\epsilon$-approximations}).

\begin{theorem}\label{thm:BB-con-CP-exp}
Given a rational $h>0$, let $P_h\subseteq \R^3$ be the convex hull of $\{(0,0,0),(0,2,0),(2,0,0),
(1-\frac{1}{h},1-\frac{1}{h},h)\}$. Consider the instances of~\eqref{eq:linear-obj} with the objective $-x_3$, $C=P_h$ and $S=\Z^3$. The optimal value is $0$ and there is a branch-and-bound algorithm with variable disjunctions which certifies $-x_3\geq 0$ in $O(1)$ steps. However, any cutting plane proof with variable disjunctions of $-x_3\geq 0$ has $\Omega(h)$ length. The instances can be created in any dimension by using the construction from Theorem~\ref{thm:BB<<CP} where one takes a Cartesian product with a hypercube.
\end{theorem}

\begin{remark}\label{rem:CP-poly}
For the case of general polytopes and general split disjunctions in fixed dimension, examples are known in which BB solves the problem in $O(1)$ iterations while any CP proof of optimality takes a number of iterations that is at least polynomial in the size of the input data. An instance of this type can be found in \cite[Lemma 19]{conforti2015reverse}: if $P\subseteq\R^3$ is the convex hull of the points $(0,0,0)$, $(2,0,0)$, $(0,2,0)$ and $(0.5,0.5,h)$, where $h>0$, and $S=\Z^3$, then the rank of the inequality $x_3\le0$ with respect to general split disjunctions is $\Omega(\log h)$.
\end{remark}

The second row of Table~\ref{tab:CP-BB-results} summarizes the results in fixed dimension. 

\section{Proofs of main results}

\subsection{Preliminaries}

We first recall with the following well-known result from LP sensitivity analysis.

{\begin{lemma}\label{lem:LP-sensitivity} Let $P = \{x\in \R^n: \langle a^i, x \rangle \leq b_i\;\; i=1, \ldots, m\}$ be a polyhedron given as the intersection of $m$ halfspaces, and let $\langle d, x \rangle \leq \delta$ be a valid inequality for $P$. Then for any $\epsilon > 0$, there exists $\epsilon' > 0$ such that $\langle d, x \rangle \leq \delta +\epsilon$ is valid for $P_{\epsilon'} := \{x\in \R^n: \langle a^i, x \rangle \leq b_i + \epsilon'\;\; i=1, \ldots, m\}$.
\end{lemma}
\begin{proof} See, for example, equation (22) in~\cite[Chapter 10]{sch}.
\end{proof}

We will need the following version of the above result.
}
\begin{lemma}\label{lem:distance} Let $C \subseteq \R^n$ be a compact, convex set. Let $a^1, \ldots, a^m, c\in \R^n$, and $b_1, \ldots, b_m,\gamma\in \R$ be such that $\langle c, x \rangle \leq \gamma$ is valid for $C\cap\bigcap_{i=1}^m \{x\in \R^n: \langle a^i, x \rangle \leq b_i\}$. For any $\epsilon > 0$, there exists $\epsilon' > 0$ such that $\langle c, x \rangle \leq \gamma + \epsilon$ is valid for $C\cap\bigcap_{i=1}^m \{x\in \R^n: \langle a^i, x \rangle \leq b_i+\epsilon'\}$.
\end{lemma}

\begin{proof} {Let $P := \{x\in \R^n: \langle a^i, x \rangle \leq b_i\;\; i=1, \ldots, m\}$. Since $\langle c, x \rangle \leq \gamma$ is valid for $C \cap P$, so is $\langle c, x \rangle \leq \gamma + \epsilon$. Therefore, if we define $\tilde C := C \cap \{x\in \R^n: \langle c, x \rangle \geq \gamma + \epsilon\}$, we have that $\tilde C \cap P = \emptyset$. $\tilde C$ is a compact, convex set as it is a closed, convex subset of a compact set. Thus, there exists a (strongly) separating hyperplane given by $d\in \R^n, \delta\in \R$ such that $\langle d,x \rangle \geq \delta$ is valid for $\tilde C$ and $\langle d,x \rangle < \delta$ for all $x \in P$ (see, for example, Problem 3 in Section III.1.3 in~\cite{barv}). By appealing to Lemma~\ref{lem:LP-sensitivity}, there exists $\epsilon' > 0$ such that $\langle d,x \rangle < \delta$ is valid for $P_{\epsilon'}$, using the notation of Lemma~\ref{lem:LP-sensitivity}. In particular, $\tilde C \cap P_{\epsilon'} = \emptyset$, i.e., $C \cap \{x\in \R^n: \langle c, x \rangle \geq \gamma + \epsilon\} \cap P_{\epsilon'} = \emptyset$. Therefore, $\langle c, x \rangle \leq \gamma + \epsilon$ is valid for $C \cap P_{\epsilon'}$, which is what we wish to establish.}
\end{proof}

\begin{lemma}\label{lem:approx-lifting} Let $C \subseteq \R^n$ be a compact, convex set. Let $\langle a, x \rangle \leq b$ be a valid inequality for $C$ that defines the face $F = C \cap \{x \in \R^n: \langle a, x \rangle = b\}$. Let $\langle c, x \rangle \leq \gamma$ be a valid inequality for $F$. Then, for any $\epsilon>0$, there exists $\lambda\geq 0$ such that $\langle c', x \rangle \leq \gamma'$ is valid for $C$ where $c' = c + \lambda a$ and $\gamma' = \gamma +\epsilon + \lambda b$.
\end{lemma}
\begin{proof} Define the set $X = \{x \in C: \langle c, x \rangle \geq \gamma+\epsilon \}$. {If $X = \emptyset$, then $\langle c, x \rangle \leq \gamma+\epsilon$ is valid for $C$ and therefore $\lambda = 0$ works. Otherwise,} note that $X$ is compact and any $\bar x \in X$ is not in $F$ and therefore $\langle a, \bar x \rangle < b$. Define $$\lambda = \max_{x \in X} \frac{\langle c, x \rangle - \gamma -\epsilon}{b - \langle a, x \rangle},$$ which is a well-defined real number because we are maximizing a continuous function (the function has strictly positive denominator for all $x\in X$ by the argument above) over a compact set. For any $\bar x \in C$, either $\bar x \in X$ or $\bar x \not\in X$. If $\bar x \in X$, by definition of $\lambda$ above, $\langle c', \bar x \rangle \leq \gamma'$. If $\bar x \not\in X$, then $\langle c', \bar x \rangle = \langle c, \bar x \rangle + \langle\lambda a, \bar x \rangle \leq \langle c, \bar x \rangle + \lambda b \leq \gamma+\epsilon + \lambda b$, where the first inequality follows from the fact that $\langle a, x \rangle \leq b$ is a valid inequality for $C$ and the second inequality follows from the fact that $\bar x \not\in X$. \end{proof}

\begin{corollary}\label{cor:approx-lifting} If $C$ is a polytope, then the above theorem holds with $\epsilon = 0$ as well.
\end{corollary}

\begin{proof} In this case, $\lambda$ can be defined by maximizing $\frac{\langle c, x \rangle - \gamma}{b - \langle a, x \rangle}$ over all vertices $x$ of $C$ not in $F$ ($\lambda$ should be defined as $0$ if the maximum is negative).\end{proof}

Motivated by Lemma~\ref{lem:approx-lifting} and Corollary~\ref{cor:approx-lifting}, we make the following definition.

\begin{definition} {Let $a, c \in \R^n$, $b, \gamma \in \R$ and $\epsilon > 0$. An inequality $\langle c', x \rangle \leq \gamma'$ is said to be an {\em $\epsilon-$approximate rotation of $\langle c, x \rangle \leq \gamma$ with respect to $\langle a, x \rangle = b$} if there exists $\lambda\geq 0$ such that $c' = c + \lambda a$ and $\gamma' \leq \gamma +\epsilon+ \lambda b$. }
\end{definition}

\begin{remark}\label{rem:rotation}If $\langle c', x \rangle \leq \gamma'$ is an $\epsilon-$approximate rotation of $\langle c, x \rangle \leq \gamma$ with respect to $\langle a, x \rangle = b$, then $\{x\in \R^n: \langle c', x \rangle \leq \gamma',\;\; \langle a,x\rangle = b\} \subseteq \{x\in \R^n: \langle c, x \rangle \leq \gamma +\epsilon,\;\; \langle a,x\rangle = b\}$. \end{remark}

\begin{theorem}\label{thm:CP-proof-lifting} Let $C \subseteq \R^n$ be a compact, convex set. Let $\langle a, x \rangle \leq b$ be a valid inequality for $C$ that defines the face $F = C \cap \{x \in \R^n: \langle a, x \rangle = b\}$. Let $S$ be any non-convexity and $\D$ be some family of valid disjunctions. Suppose $\langle c, x \rangle \leq \gamma$ is a valid inequality for $F \cap S$ and $\langle a_1, x \rangle \leq b_1, \ldots, \langle a_N, x \rangle \leq b_N$ is a cutting plane proof of $\langle c, x \rangle \leq \gamma$ based on $\D$ (with respect to $F \cap S$). Then, for any $\bar\epsilon > 0$, there exists a sequence of inequalities $\langle a_1', x \rangle \leq b_1', \ldots, \langle a_N', x \rangle \leq b_N'$ and an inequality $\langle c', x \rangle \leq \gamma'$ such that all of the following hold.
\begin{enumerate}
\item For each $i=1, \ldots, N$, $\langle a_i', x \rangle \leq b_i'$ is an $\bar\epsilon$-approximate rotation of $\langle a_i, x \rangle \leq b_i$ with respect to {$\langle a, x \rangle = b$}.
\item $\langle c', x \rangle \leq \gamma'$ is an $\bar\epsilon$-approximate rotation of $\langle c, x \rangle \leq \gamma$ with respect to {$\langle a, x \rangle = b$}.
\item $\langle a_1', x \rangle \leq b_1', \ldots, \langle a_N', x \rangle \leq b_N'$ is a cutting plane proof of $\langle c', x \rangle \leq \gamma'$ based on $\D$, with respect to $C \cap S$.
\end{enumerate}

\end{theorem}

\begin{proof} We prove the theorem by induction on the length $N$ of the CP proof. If $N=0$, then $\langle c, x \rangle \leq \gamma$ is a valid inequality for $F$ itself and the result follows from Lemma~\ref{lem:approx-lifting}. Consider $N\geq 1$. Fix an arbitrary $\bar \epsilon > 0$ as the ``error parameter". Apply Lemma~\ref{lem:distance} with $C=F$, $m = N$, $a^1, \ldots, a^N, c$, $b_1, \ldots, b_N, \gamma$ and $\epsilon = \frac12\bar\epsilon$ to get $\epsilon' > 0$. Set $\hat\epsilon = \frac12\min\{\epsilon',\bar\epsilon\}$.

Let $D = Q_1 \cup \ldots \cup Q_k\in \D$ be the disjunction used to derive $\langle a_N, x \rangle \leq b_N$ for $F \cap \bigcap_{i=1}^{N-1} \{x \in \R^n: \langle a_i, x \rangle \leq b_i\}$. For each $j=1, \ldots, k$, let $\epsilon_j > 0$ be defined by applying Lemma~\ref{lem:distance} with $C=Q_j \cap F$, $m = N-1$, $a^1, \ldots, a^{N-1}$, $c = a^N$, $b_1, \ldots, b_{N-1}$, $\gamma = b_N$ and $\epsilon = \hat \epsilon$. Define $\epsilon^\star = \min\{\hat\epsilon,\min_{j=1}^k \epsilon_j\}$.

By the induction hypothesis applied to the CP proof $\langle a_1, x \rangle \leq b_1, \ldots, \langle a_{N-1}, x \rangle \leq b_{N-1}$, viewed as a proof of $\langle a_{N-1}, x \rangle \leq b_{N-1}$, with $\epsilon = \epsilon^\star$, there exists a sequence of inequalities $\langle a_1', x \rangle \leq b_1', \ldots, \langle a_{N-1}', x \rangle \leq b_{N-1}'$ such that $\langle a_i', x \rangle \leq b_i'$ is an $\epsilon^\star$-approximate rotation of $\langle a_i, x \rangle \leq b_i$ for each $i=1, \ldots, N-1$ and $\langle a_1', x \rangle \leq b_1', \ldots, \langle a_{N-1}', x \rangle \leq b_{N-1}'$ is a CP proof of $\langle a_{N-1}', x \rangle \leq b_{N-1}'$ based on $\D$, with respect to $C \cap S$ {(if $N=1$, then we consider the trivial CP proof of length $0$ for the trivial inequality $\langle 0, x \rangle \leq 1$.)}

By Remark~\ref{rem:rotation}, $F\cap \{x \in \R^n: \langle a_i', x \rangle \leq b_i'\} \subseteq F\cap \{x \in \R^n: \langle a_i, x \rangle \leq b_i +\epsilon^\star\}$. Since $\langle a_N, x \rangle \leq b_N$ is valid for $D \cap F \cap \bigcap_{i=1}^{N-1} \{x \in \R^n: \langle a_i, x \rangle \leq b_i\}$, it is valid for $Q_j \cap F \cap \bigcap_{i=1}^{N-1} \{x \in \R^n: \langle a_i, x \rangle \leq b_i\}$ for every $j=1, \ldots, k$. By Lemma~\ref{lem:distance} and the choice of $\epsilon^\star$, $\langle a_N, x \rangle \leq b_N + \hat\epsilon$ is valid for $Q_j \cap F\cap \bigcap_{i=1}^{N-1}\{x \in \R^n: \langle a_i', x \rangle \leq b_i'\}$ for every $j=1, \ldots, k$. Since $F$ is a face of $C$ induced by $\langle a, x \rangle \leq b$, we have that $Q_j \cap F\cap \bigcap_{i=1}^{N-1}\{x \in \R^n: \langle a_i', x \rangle \leq b_i'\}$ is a face of $Q_j \cap C \cap \bigcap_{i=1}^{N-1}\{x \in \R^n: \langle a_i', x \rangle \leq b_i'\}$ induced by the same inequality. By Lemma~\ref{lem:approx-lifting}, there exists an $\hat\epsilon$-approximate rotation of $\langle a_N, x \rangle \leq b_N + \hat\epsilon$ valid for $Q_j \cap C \cap \bigcap_{i=1}^{N-1}\{x \in \R^n: \langle a_i', x \rangle \leq b_i'\}$ for every $j=1, \ldots, k$. In other words, there exist $\lambda_j \geq 0$, $j=1, \ldots, k$ such that
$\langle a_N + \lambda_j a, x \rangle \leq (b_N + \hat\epsilon) + \hat\epsilon + \lambda_j b$ is valid for $Q_j \cap C \cap \bigcap_{i=1}^{N-1}\{x \in \R^n: \langle a_i', x \rangle \leq b_i'\}$.
Set $\lambda = \max_{j=1}^n \lambda_j$, $a_N' = a_N + \lambda a$ and $b_N' = b_N+2\hat\epsilon + \lambda b$. Thus, $\langle a_N', x \rangle \leq b_N'$ is valid for $Q_j \cap C \cap \bigcap_{i=1}^{N-1}\{x \in \R^n: \langle a_i', x \rangle \leq b_i'\}$ for all $j=1, \ldots, k$, and therefore for $D \cap C \cap \bigcap_{i=1}^{N-1}\{x \in \R^n: \langle a_i', x \rangle \leq b_i'\}$. Since $2\hat\epsilon \leq \bar\epsilon$ by choice, $\langle a_N', x \rangle \leq b_N'$ is an $\bar\epsilon$-approximate rotation of $\langle a_N, x \rangle \leq b_N$, and thus condition 1. is satisfied for $i=1, \ldots, N$.

From the hypothesis, $\langle c, x \rangle \leq \gamma$ is valid for $F\cap \bigcap_{i=1}^{N}\{x \in \R^n: \langle a_i, x \rangle \leq b_i\}$. The definition of $\epsilon'$ implies that $\langle c, x \rangle \leq \gamma + \frac{\bar\epsilon}{2}$ is valid for $F\cap \bigcap_{i=1}^{N}\{x \in \R^n: \langle a_i, x \rangle \leq b_i + \epsilon'\}$. Also, by choice $\epsilon^\star \leq \hat\epsilon \leq 2\hat\epsilon \leq \epsilon'$ and so $F\cap \bigcap_{i=1}^{N}\{x \in \R^n: \langle a_i, x \rangle \leq b_i + \epsilon'\} \supseteq F\cap \bigcap_{i=1}^{N}\{x \in \R^n: \langle a_i, x \rangle \leq b_i + 2\hat\epsilon\} \supseteq F\cap \bigcap_{i=1}^{N}\{x \in \R^n: \langle a_i', x \rangle \leq b_i'\}$, where the second containment follows from Remark~\ref{rem:rotation}. Since $F\cap \bigcap_{i=1}^{N}\{x \in \R^n: \langle a_i', x \rangle \leq b_i'\}$ is a face of $C\cap \bigcap_{i=1}^{N}\{x \in \R^n: \langle a_i', x \rangle \leq b_i'\}$ induced by $\langle a, x \rangle \leq b$, by Lemma~\ref{lem:approx-lifting} with $\epsilon = \frac{\bar\epsilon}{2}$,
there exists $\hat\lambda\geq 0$ such that $\langle c', x \rangle \leq \gamma'$ is valid for $C\cap \bigcap_{i=1}^{N}\{x \in \R^n: \langle a_i', x \rangle \leq b_i'\}$, with $c' = c + \hat\lambda a$ and $\gamma' = (\gamma + \frac{\bar\epsilon}{2}) + \frac{\bar\epsilon}{2}+ \hat\lambda b = \gamma +\bar\epsilon + \hat\lambda b$. Thus, conditions 2. and 3. of the theorem are satisfied.
\end{proof}

\begin{corollary}\label{cor:rotation-poly} If $C$ is a polytope, the statement of Theorem~\ref{thm:CP-proof-lifting} holds with $\bar\epsilon = 0$.
\end{corollary}
\begin{proof}
This is handled by appealing to Corollary~\ref{cor:approx-lifting} instead of Lemma~\ref{lem:approx-lifting} in the proof of Theorem~\ref{thm:CP-proof-lifting}. There will be no need for Lemma~\ref{lem:distance} anymore in this case.
\end{proof}

The above proof is inspired by ideas in Dash~\cite{dash2002exponential}, where the polyhedral case is analyzed. In the general convex case, faces may not be exposed and the above proof deals with this issue.

\subsection{Proofs of the main results}
\begin{proof}[Proof of Theorem~\ref{thm:CP<=BB}] 
We prove this by induction on the number $k$ of branching nodes in the BC proof. If $k=0$, then we have a CP proof and we are done. Now consider a BC proof with $k\geq 1$ branching nodes. Note that all nodes in the tree represent subsets of $C$ obtained by intersecting $C$ with additional halfspaces (either cutting planes or disjunction inequalities of the form $x_i \leq 0$ or $x_i \geq 1$), i.e., each node in the tree is a compact, convex subset of $C$.
Consider any {\em maximal depth} branching node, that is, all its descendants are cutting nodes or leaves. Suppose this branching node represents a compact, convex subset $C' \subseteq C$ and uses the disjunction $\{x: x_i \leq 0\} \cup \{x: x_i \geq 1\}$. Let the two children of $C'$ be $C_1 = C' \cap \{x: x_i \leq 0\}$ and $C_2 = C' \cap \{x: x_i \geq 1\}$, which are both faces of $C'$.

Let $\epsilon' =\frac{\epsilon}{2}$. Since $C'$ corresponds to a maximal depth branching node, the BC proof under consideration yields a CP proof of $\langle c, x\rangle \leq \gamma$ with respect to $C_1\cap S$. Using Theorem~\ref{thm:CP-proof-lifting} with $\bar\epsilon = \epsilon'$, we can find a CP proof of $\langle c_1, x\rangle \leq \gamma_1$ with respect to $C'\cap S$ such that $\langle c_1, x\rangle \leq \gamma_1$ is an $\epsilon'$-approximate rotation of $\langle c, x\rangle \leq \gamma$ with respect to {$x_i = 0$}. In other words, there exists $\lambda_1 \geq 0$ such that $c_1 = c+ \lambda_1 e^i$ and $\gamma_1\leq \gamma + \epsilon'$, where $e^i$ is the $i$-th standard unit vector. Similarly, applying Theorem~\ref{thm:CP-proof-lifting} to the CP proof of $\langle c, x\rangle \leq \gamma$ with respect to $C_2\cap S$ (the other branch obtained from $x_i \geq 1$), we can find a CP proof of $\langle c_2, x\rangle \leq \gamma_2$ with respect to $C'\cap S$ such that $\langle c_2, x\rangle \leq \gamma_2$ is an $\epsilon'$-approximate rotation of $\langle c, x\rangle \leq \gamma$ with respect to {$x_i=1$}. In other words, there exists $\lambda_2 \geq 0$ such that $c_2 = c + \lambda_2 (- e^i)$ and $\gamma_2 \leq \gamma + \epsilon' - \lambda_2$.

Consider the set $C''$ obtained by intersecting $C'$ with the two inequalities $\langle c_1, x \rangle \leq \gamma_1$ and $\langle c_2, x \rangle \leq \gamma_2$. Now observe that if we consider the face {$F_1$ of $C''$ defined by $x_i = 0$}, the inequality $\langle c_1, x \rangle \leq \gamma_1$ reduces to $\langle c, x \rangle \leq \gamma+\epsilon'$, { i.e., $F_1 \cap \{x : \langle c_1, x \rangle \leq \gamma_1\} = F_1 \cap \{x: \langle c, x \rangle \leq \gamma+\epsilon'\}$. Similarly, on the face defined by $x_i = 1$, the inequality $\langle c_2, x \rangle \leq \gamma_2$ also reduces to $\langle c, x \rangle \leq \gamma + \epsilon'$. Thus, $\langle c, x \rangle\leq \gamma + \epsilon'$ is valid for both these faces of $C''$.} Thus, we can derive this inequality as a cutting plane for $C''$ using the disjunction $\{x: x_i \leq 0\} \cup \{x: x_i \geq 1\}$. Thus, {by concatenating the CP proofs of $\langle c_1, x \rangle \leq \gamma_1$ and $\langle c_2, x \rangle \leq \gamma_2$ and then deriving $\langle c, x \rangle \leq \gamma + \epsilon'$ using the disjunction $\{x: x_i \leq 0\} \cup \{x: x_i \geq 1\}$, we have replaced the entire tree below $C'$ with a CP proof such that $\langle c, x \rangle \leq \gamma + \epsilon'$ is valid for the leaf. Moreover, the length of this CP proof is exactly one less than the number of nodes below $C'$ in the original branch-and-cut tree (since we do not have the two branching nodes, but have an extra cutting plane derivation). Thus, this replacement gives us a new BC proof of $\langle c, x \rangle \leq \gamma + \epsilon'$ with one less branching node.} We now appeal to the induction hypothesis with $\epsilon = \epsilon'$ for the new, modified BC proof of $\langle c, x \rangle \leq \gamma + \epsilon'$. Thus, we obtain a CP proof of size at most the new BC proof (whose size is at most $N$, the size of the original BC proof) for the inequality $\langle c, x \rangle \leq (\gamma + \epsilon') + \epsilon'$. By choice, $2\epsilon' = \epsilon$, we have thus produced a CP proof of $\langle c, x \rangle \leq \gamma + \epsilon$ with size at most $N$.

The case when $C$ is a polytope is handled by appealing to Corollary~\ref{cor:rotation-poly} in the above proof.
\end{proof}

\begin{proof}[Proof of Theorem~\ref{thm:CP<<BB}] We will use the fact that, for a clique with $t\ge2$ vertices, the LP value of the independent set problem is $t/2$ (and 1 if $t=1$).

We first prove a lower bound on the size of any BB proof that establishes the upper bound of $m$ on the objective value.

\begin{claim}
For $m\geq 1$, any branch-and-bound proof of $\sum_{v} x_v \leq m$ has size at least $2^{m+1}-2$.
\end{claim}

\begin{proof} Consider first any feasible node of the branch-and-bound tree at depth $k < m$. Consider the path from the root node to this node at depth $k$ and suppose that of $k$ nodes on this path that were branched on, $k_1$ were set to $1$ and $k_2$ were set to $0$ (so $k_1 + k_2 = k$). Feasibility implies that all the $k_1$ vertices set to $1$ are from different cliques. From the remaining $m - k_1$ cliques, $k_2$ vertices were set to $0$, so the LP value at this node is at least $k_1 + \frac{3(m-k_1) - k_2}{2} = \frac{3m - (k_1 + k_2)}{2} = \frac{3m - k}{2}$. Since we assumed $k < m$, the LP has value strictly bigger than $m$. Thus, if the branch-and-bound tree has only feasible nodes, only nodes at depth $m$ or more can have LP values at most $m$. Since every branching node has at least two children, this means we have at least $2^{m+1} - 1$ nodes in the tree, thus the size of the proof is at least $2^{m+1}-2$.

Now, assume that some node $N_1$ of the BB tree is infeasible, which means at the node $N_1$, we have set two variables, which are denoted as $x_1$ and $x_2$, of the same clique to $1$. Let $N_0$ be the parent node of $N_1$, and let $N_2$ be the other child of $N_0$. In $N_2$, we have $x_1=1$ and $x_2=0$ (or vice versa). However, $x_2=0$ is a redundant constraint once $x_1=1$ is imposed since they belong to the same clique. Thus, $N_0$ and $N_2$ have the same LP objective value. So we can eliminate the node $N_1$, and contract $N_0$ and $N_2$ to obtain an equivalent BB tree with fewer nodes. By doing the same thing to every infeasible node, we can get a new BB tree with fewer nodes than the original BB tree. Thus, the smallest size BB proofs are those with only feasible nodes. And we proved a lower bound of $2^{m+1}-2$ on the size of such BB proofs above.\end{proof}

Next we show that there is an efficient cutting plane algorithm that solves the problem.
\begin{claim}\label{claim:CP-fast}
There is a cutting plane algorithm based on variable disjunctions that solves the problem in $m$ iterations.
\end{claim}
\begin{proof} $P(K_3) := \{(x_u, x_v, x_w) \in \R^3_+: \;\; x_u + x_v \leq 1,\;\; x_u + x_w \leq 1,\;\; x_v + x_w \leq 1\}$. Applying the disjunction $x_u \leq 0$ and $x_u \geq 1$, we see that $x_u + x_v + x_w \leq 1$ is valid for this disjunction.  Thus, for each copy of $K_3$, we derive this inequality in one iteration. Furthermore, this inequality cuts off the optimal solution of the LP relaxation, as the optimal LP value is $3/2$. Therefore in $m$ iterations, we have the inequality that bounds the sum of all variables by $m$.\end{proof}
The two claims above together complete the proof of Theorem~\ref{thm:CP<<BB}.
\end{proof}

\begin{proof}[Proof of Theorem~\ref{thm:BB<<CP}] We argue the $n=2$ case. The general $n$ dimensional case is similar. Consider applying cutting planes derived from variable disjunctions on the convex hull of $\{(0,0), (1.5,1), (2,2), (1, 1.5)\}$. By symmetry, one can focus on the behavior of any one of the two fractional vertices. It is not hard to show that the best one can do is to ``move" the vertex $(1.5,1) = (1 + \frac{1}{2},1)$ to $(1 + \frac{1}{4},1)$ in two iterations; and after $2K$ iterations, the vertex moves to $(1 + \frac{1}{2+2^K},1)$; see Figure~\ref{fig:BB<<CP}. Thus, to maximize $x_1 - x_2$, any CP algorithm takes infinitely many iterations to reach the optimal value of $0$. However, after just two BB steps based on the disjunctions $D_{1,1}$ and $D_{2,0}$ (recall the notation from~\eqref{eq:var-disj}), the leaves of the BB tree consist of the feasible integer points $(2,2)$ and $(0,0)$, and a triangle over which the LP relaxation has $(1,1)$ as its unique optimal solution.
\end{proof}

{\begin{proof}[Proof of Theorem~\ref{thm:fixed-dim-0/1}] For a BB algorithm (see Definition~\ref{def:BB-algo}), one simply iterates through the disjunctions $\{x: x_i \leq 0\}\cup \{x: x_i \geq 1\}$, $i=1, ..., n$ in order and applies the disjunction that cuts off the optimal solution to $\max\{\langle c, x \rangle: x \in N\}$ at every node $N$ (the nodes can be selected arbitrarily in $\mathcal{L}$). In the worst case, one enumerates all integer points and so the algorithm takes at most $O(2^n)$ iterations.

For a CP proof, one first observes that since $C$ is compact, for every $0/1$ point not in $C$, one can separate it from $C$ by a separating hyperplane. Thus, there exists a polytope $P$ such that $C \subseteq P$ and $P \cap \Z^n = C \cap \Z^n$. Balas' sequential convexification theorem~\cite{balas1998disjunctive} implies that, starting from $P$, if we repeatedly convexify with respect to the disjunctions $\{x: x_i \leq 0\}\cup \{x: x_i \geq 1\}$, $i=1, ..., n$, in order, we obtain the polytope $\conv(P \cap \Z^n) = \conv(C \cap \Z^n)$. Let us label this sequence of polytopes as $$P_0 = P,\;\;\ldots, \;\;P_i = \conv(P_{i-1}\cap\{x: x_i \leq 0\} \cup P_{i-1} \cap \{x: x_i \geq 1\}),\;\; \ldots, \;\;P_n = \conv(C \cap \Z^n).$$ Moreover, since we started with $C \subseteq P$, any sequence of cutting planes derived for $P$ is a valid sequence of cutting planes for $C$. By a result of Andersen, Cornu\'ejols and Li~\cite{DBLP:journals/mp/AndersenCL05}, any split inequality for a polytope in $\R^n$ is also a split inequality for a relaxation of the polytope obtained from at most $n$ linearly independent valid inequalities. By combining this result with Balas' theorem and Carath\'eodory's theorem, it is not hard to get a CP proof that starts from $P$ and proves the validity of $\langle c, x \rangle \leq \gamma$ whose length only depends on the dimension $n$. This same proof can be used starting from $C$.
\end{proof}}

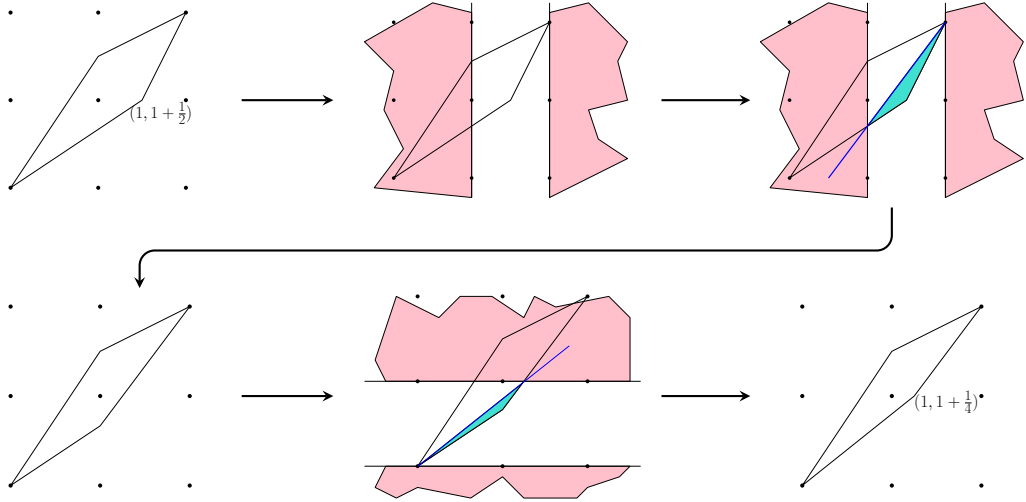
\begin{figure}
\centering
\begin{tikzpicture} [ font = \small, align = flush center, >=stealth, thick,
  %node distance = 4cm,
  %inner xsep = 0pt,
  every node/.style={rounded corners, text width=3.5cm, minimum height=1.3cm}]
  \node [fill=white!20] (initialize)
  {\includegraphics[width=0.7\textwidth]{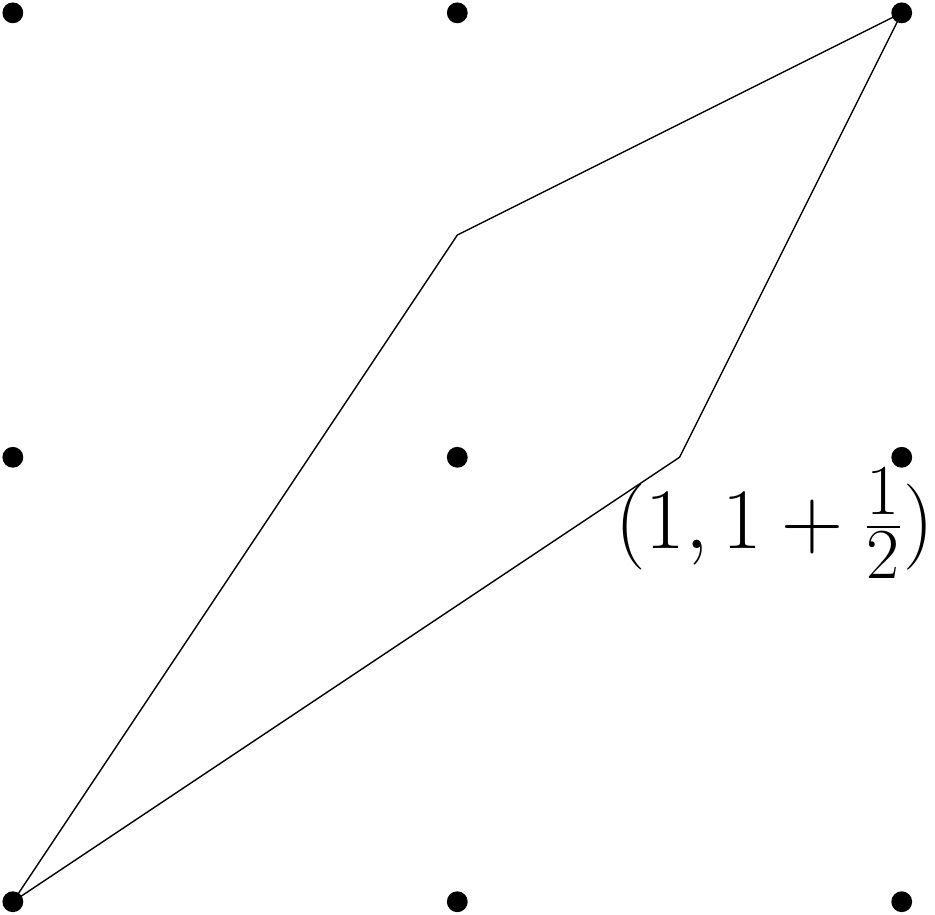}};
  \node [fill=white!20, right = 1.5cm of initialize] (cut1)
  {\includegraphics[width=\textwidth]{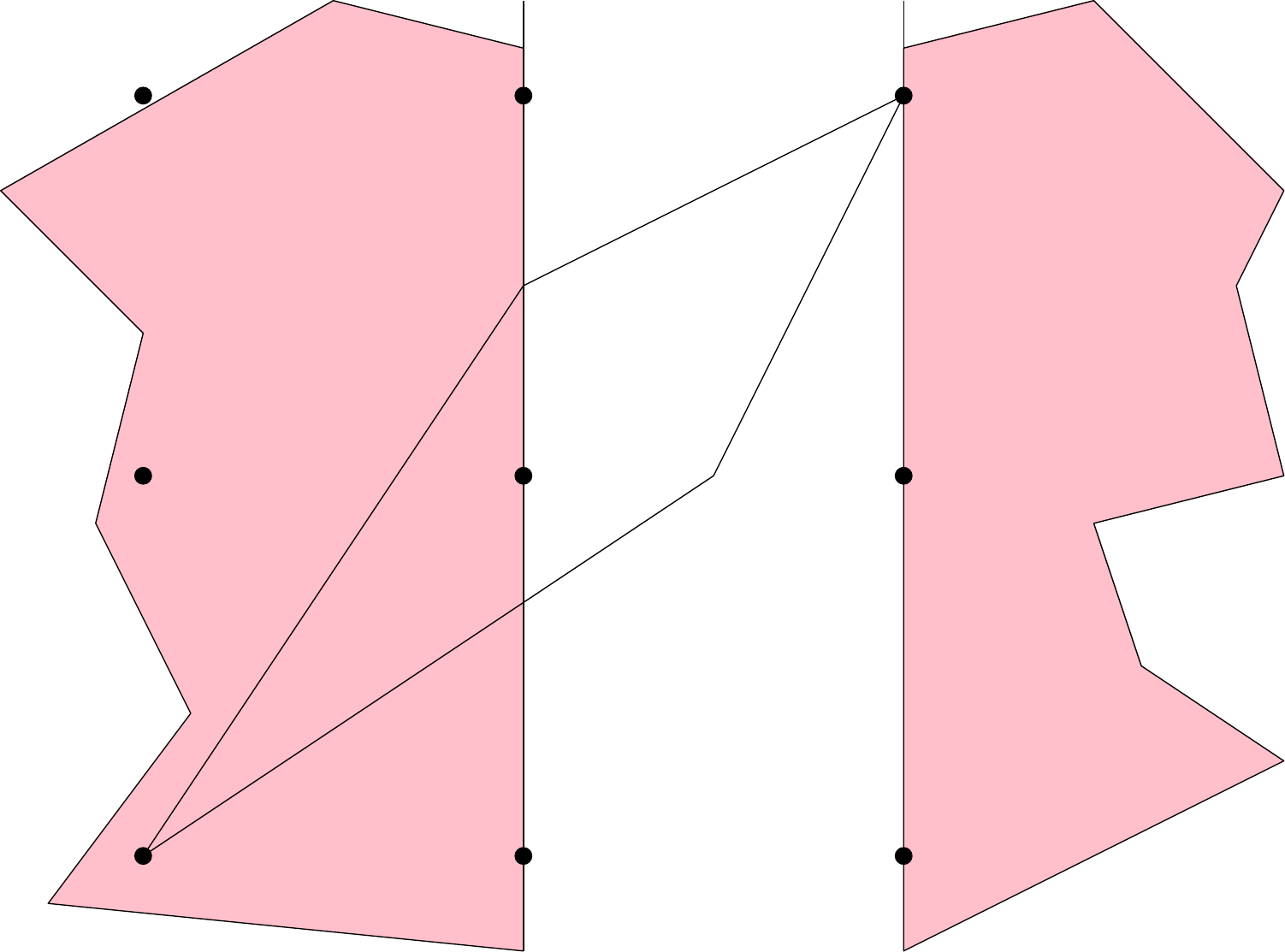}};
  \node [fill=white!20, right = 1.5cm of cut1] (cut2)
  {\includegraphics[width=\textwidth]{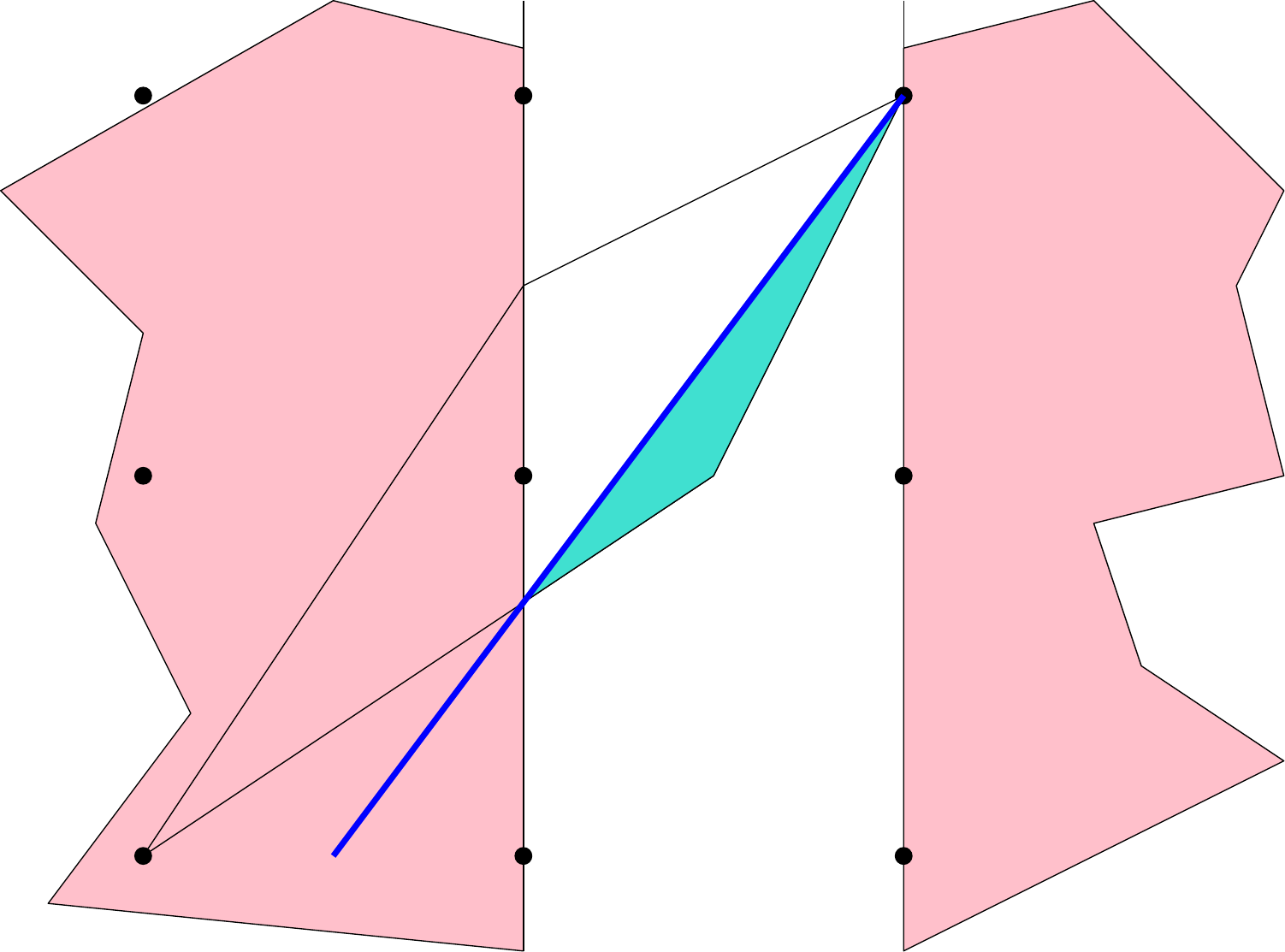}};
 \node [fill=white!20, below = 1.25cm of initialize] (cut3)
  {\includegraphics[width=0.7\textwidth]{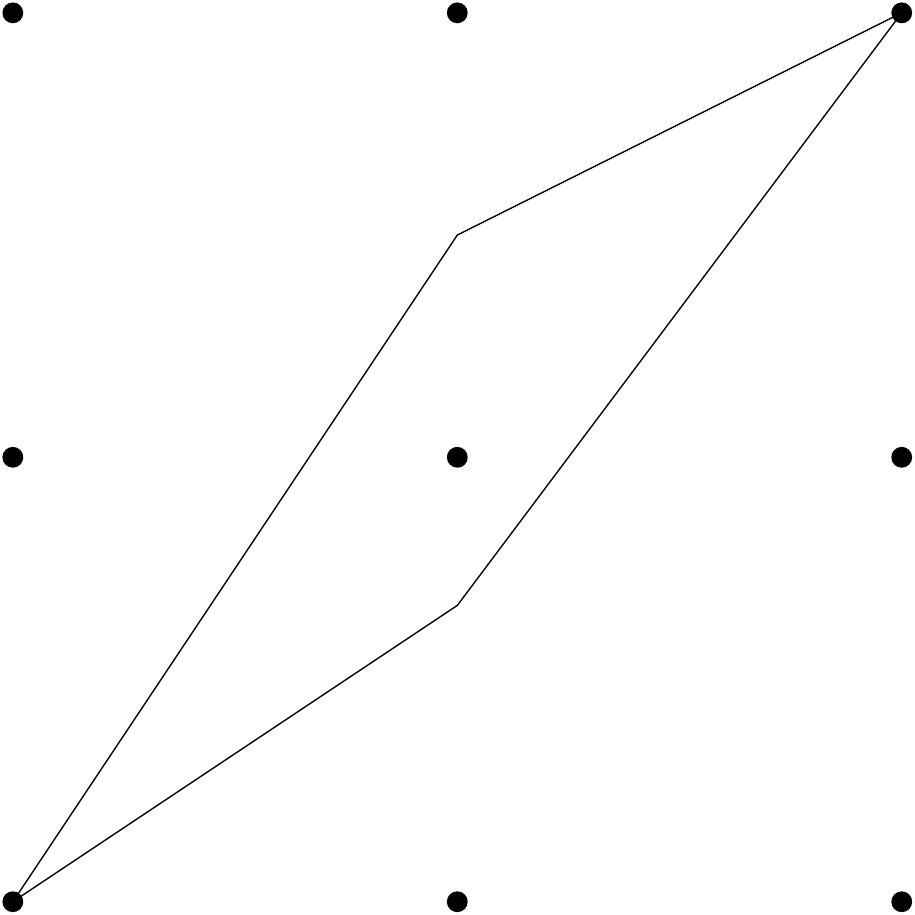}};
  \node [fill=white!20, right = 1.5cm of cut3] (cut4)
  {\includegraphics[width=1.05\textwidth]{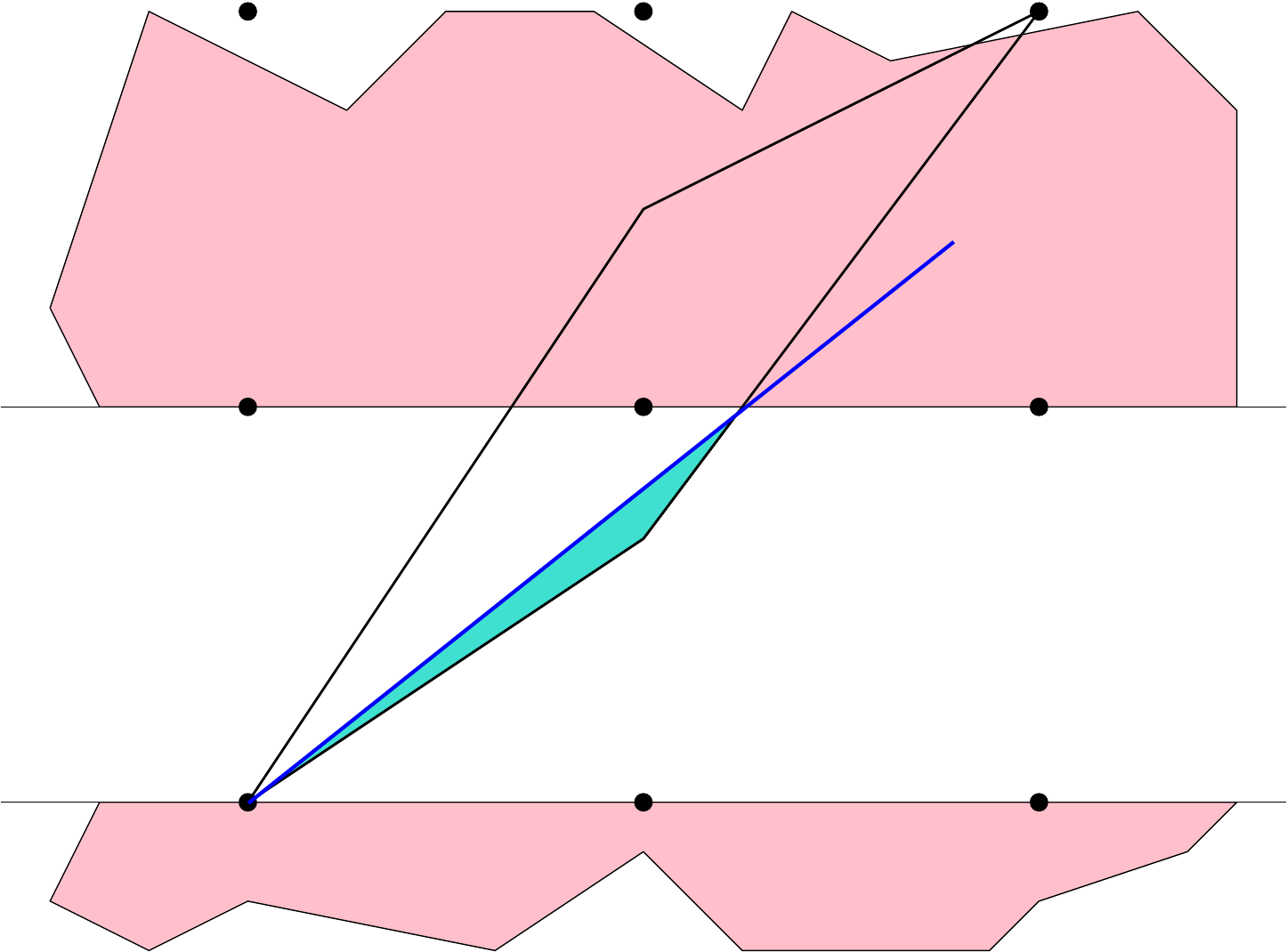}};
  \node [fill=white!20, right = 1.5cm of cut4] (cut5)
  {\includegraphics[width=0.7\textwidth]{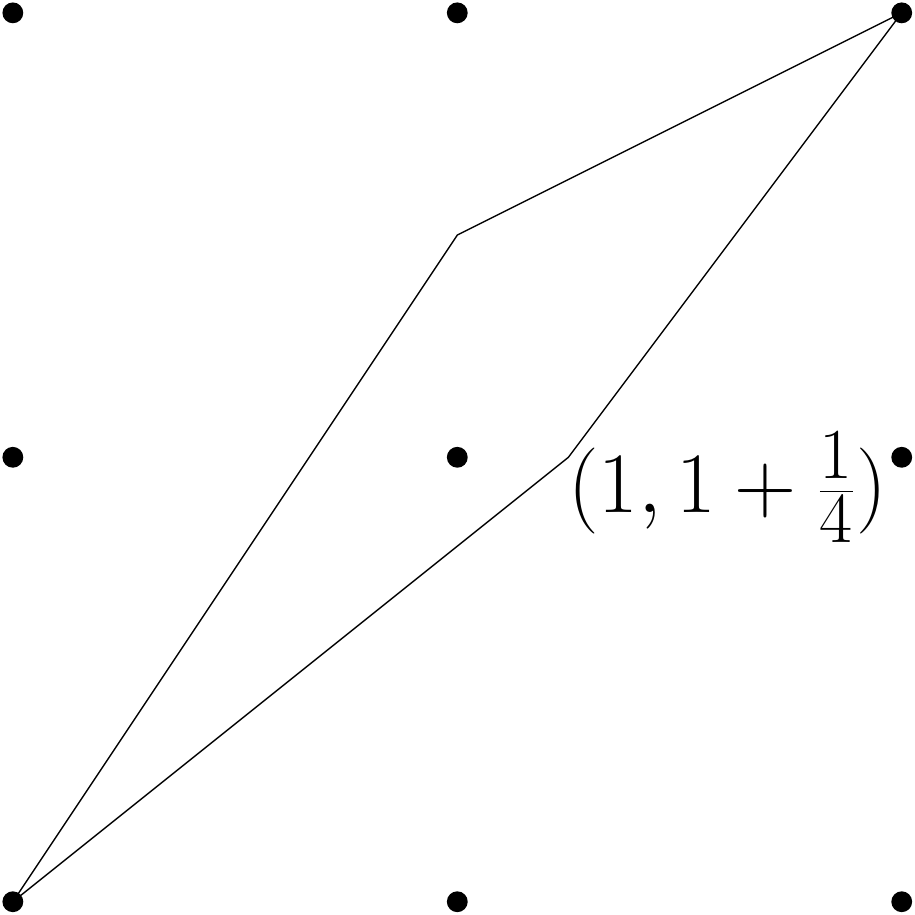}};
  \draw[->] (initialize) -- +(3.1,0);
  \draw[->] (cut1) +(2.2,0) -- (cut2);
  \draw[->,rounded corners=2mm] (cut2) -- +(0, -2.0) -- +(-10.0,-2.0) -- +(-10.0,-2.5);
  \draw[->] (cut3) -- +(3.1,0);
  \draw[->] (cut4) +(2.2,0) -- (cut5);
 \end{tikzpicture}
 \caption{Cutting plane proofs based on variable disjunctions can have infinite length.}\label{fig:BB<<CP}
 \end{figure}

\begin{proof}[Proof of Theorem~\ref{thm:BB<=CP}] {Let us denote by $D_1, \ldots, D_K \in \D$ the sequence of disjunctions (with possible repetitions) used by the CP proof.} Denote the relaxation after applying the cuts based on $D_1, \ldots, D_i$ as $C_i$, i.e., the relaxation after the $i$-th {step} of the CP proof. Define $z^{CP}(i): = \sup\{\langle c, x\rangle : x \in C_i\}$, $i=1, \ldots, K$. Thus, $z^{CP}(K) \leq \gamma$.

We ``simulate" the CP proof with a BB procedure as follows.
\begin{enumerate}
\item Initialize $\cL := \{C\}$.
\item For $i=1,2, \ldots, K$, do
\begin{enumerate}
\item Compute $z^{BB} := \max_{Q \in \cL}\left\{\sup\{\langle c, x\rangle : x \in Q\}\right\}$ (we use the convention that $\sup\{\langle c, x\rangle : x \in Q\} = -\infty$ if $Q =\emptyset$).
\item While $z^{BB} > z^{CP}(i)$ do
\begin{itemize} \item Let $Q$ be a node in $\cL$ that achieves the maximum in the definition of $z^{BB}$ and let $v \in Q$ be any point such that $\langle c, v\rangle > z^{CP}(i)$. Clearly, {$v \not\in C_i$} and must be cut off by some inequality generated by one of the disjunctions $D_1, \ldots, D_i$. Let this disjunction be $D^\star =(Q^\star_1 \cup \ldots \cup Q^\star_k)$. Apply this disjunction on $Q$, i.e., replace $Q$ by $\bigcup_{i=1}^k Q^\star_i\cap Q$ in $\cL$.
\end{itemize}
\end{enumerate}
\end{enumerate}

Observe that if the above algorithm stops in a finite number of steps, then $z^{BB} \leq z^{CP}(K) \leq \gamma$. Thus we have a BB proof of $\langle c, x \rangle \leq \gamma$. Furthermore, the above BB procedure is a BB algorithm in the sense of Definition~\ref{def:BB-algo} if at every iteration $v$ can be chosen to be a point attaining the supremum, i.e., a point such that $\langle c,v\rangle=\sup\{\langle c, x\rangle : x \in Q \}$, where $Q$ is the node selected in step 2(b).

We now show that the BB algorithm generates a BB proof tree of size at most $O((MK)^{n+1})$.

Consider the tree at the end of the BB algorithm. The number of feasible {\em leaves} is at most the number of nonempty cells in a hyperplane arrangement with $MK$ hyperplanes since each $D_j$ has complexity at most $M$. By standard combinatorics of hyperplane arrangements~\cite{zaslavsky1977combinatorial,buck1943partition}, there can be at most $O((M K)^n)$ cells which then bounds the number of feasible leaves generated. Since every node is at distance at most $K$ from the root (because no disjunction is applied more than once along the same path from the root to a leaf), there are at most $O(K(MK)^n)$ feasible {\em nodes} in the entire tree. Note that only feasible nodes in the tree can have children because if a node in the tree is infeasible, no disjunction is ever applied to it in any BB algorithm, and certainly not in the above one. Since each feasible node of the tree has at most $M$ children because the complexity of each disjunction is at most $M$, the total number of nodes in the tree is bounded by $O((MK)^{n+1})$.\end{proof}

\begin{proof}[Proof of Theorem~\ref{thm:BB-algorithm-fixed-dim}] Since $\langle c, x \rangle \leq \gamma$ is valid for $C \cap \Z^n$, the set $C \cap \{x\in \R^n: \langle c, x \rangle \geq \lceil\gamma\rceil\}$ has no integer points. Consider a BB proof that first branches on $\langle c, x \rangle \leq \lfloor\gamma\rfloor$ and $\langle c, x \rangle \geq \lceil\gamma\rceil$.  By the flatness theorem~\cite{BanaszczykLitvakPajorSzarek99,banaszczyk1996inequalities,rudelson2000distances}, there is a BB proof based on split disjunctions that proves infeasibility of the latter branch with size at most $f(n)^n$. Using arguments similar to those used in the proof of Theorem~\ref{thm:BB<=CP}, we can convert this BB proof into a BB algorithm that takes at most $O((2f(n))^{n(n+1)})$ iterations.
\end{proof}

\begin{proof}[Proof of Theorem~\ref{thm:BB-con-CP-exp}] For any $h'>1$, let $v^1(h')=(1-\frac{1}{h},1-\frac{1}{h},h')$ and $v^2(h')=(1-\frac{1}{h},1-\frac{1}{h},h'\cdot\frac{h-1}{h+1})$. Consider any polyhedron $Q$ that contains $\conv\{(0,0,0), (0,2,0), (2,0,0), v^1(h')\}$.

We first claim that every split cut for $Q$ based on variable disjunctions on the variables $x_1$ and $x_2$ is valid for $v^2(h')$. For any such disjunction other than $D_{1,0}$ and $D_{2,0}$, $v^2(h')$ is contained in one of the halfspaces of the disjunction, so the claim is true. Consider $D_{1,0}$. Let $v^3(h'):=(1,\frac{h-1}{h+1}, h'\cdot \frac{h}{h+1})$. Clearly $v^3(h')\in D_{1,0}$. Since $v^3(h')$ lies on the segment between $v^1(h')$ and $(2,0,0)$, we have $v^3(h')\in Q$. Thus $v^3(h')\in Q\cap D_{1,0}$. Since $v^2(h')$ is in the segment between $v^3(h')$ and $(0,\frac{2(h-1)}{h+1},0) \in Q \cap D_{1,0}$, we have that $v^2(h')$ is in the convex hull of $Q\cap D_{1,0}$. With a symmetric argument, it can be proved that  a cutting plane derived from $D_{2,0}$ cannot cut off $v^2(h')$ from $Q$. Thus, the statement follows.

It is also not hard to see that any split cut for $Q$ based on a disjunction on $x_3$ is valid for $(1-\frac{1}{h},1-\frac{1}{h},h'-1)$.

We conclude that the closure of all split cuts based on variable disjunctions for $Q$ must contain the point $(1-\frac{1}{h},1-\frac{1}{h},\min\{h'\cdot\frac{h-1}{h+1}, h'-1\}) = (1-\frac{1}{h},1-\frac{1}{h},\min\{h' - h'\cdot\frac{2}{h+1}, h'-1\}) = (1-\frac{1}{h},1-\frac{1}{h},h' - \max\{h'\cdot\frac{2}{h+1}, 1\})$.
\medskip

Let $k$ be the length of any CP proof based on variable disjunctions with respect to $P(h)$ to certify the validity of the inequality $x_3\leq (h+1)/2$. When $h'\geq \frac{h+1}{2}$,  we have $h'\cdot(\frac{2}{h+1})\geq 1$. So from the above arguments, $k$ must satisfy

\begin{equation*}
\begin{array}{rccl}
    & h\cdot \left(1-\frac{2}{h+1}\right)^k  &\leq &\frac{h+1}{2}\\
    \Rightarrow & 1-k\frac{2}{h+1} &\leq &\frac{1}{2} + \frac{1}{2h} \\
     \Rightarrow &\frac{h+1}{4} - \frac{h+1}{4h} & \leq & k
    \end{array}
\end{equation*}
where the second inequality follows from the fact that $1-nx \leq (1-x)^n$, assuming $0 < x < 1$, i.e., assuming $h > 1$.
Thus, for $h > 1$, to certify $x_3\leq 0$ from $x_3\leq h$, we need at least $\Omega(h)$ steps.

However, in just two BB steps based on the disjunctions $D_{1,0}$ and $D_{2,0}$, we can obtain the optimal value.
\end{proof}

\subsection{Some additional remarks}\label{sec:misc}

Theorem~\ref{thm:CP<=BB} that shows cutting planes do just as well as branch-and-bound for 0/1 problems can be generalized a little to the following scenario.

\begin{theorem}\label{thm:CP<=BB-II} Let $C\subseteq \R^n$ be a compact, convex set. Let $\D$ be a family of valid disjunctions for some non-convexity $S\subseteq \R^n$, such that each disjunction in $\D$ is {\em facial} for $C$, i.e., $C \cap D$ is a union of exposed faces of $C$.

Let $\langle c, x \rangle \leq \gamma$ be a valid inequality for $C \cap S$ (possibly $c=0$, $\gamma = -1$ if $C \cap S = \emptyset$). If there exists a branch-and-cut proof of size $N$ based on $\D$ that certifies $\langle c, x \rangle \leq \gamma$, then for any $\epsilon > 0$, there exists a cutting plane proof based on $\D$ of size at most $N$ certifying $\langle c, x \rangle \leq \gamma + \epsilon$.

If $C$ is a polytope, then the statement is also true with $\epsilon = 0$.
\end{theorem}

\begin{proof} The above result can be established by simply observing that in the proof of Theorem~\ref{thm:CP<=BB}, instead of two faces arising out of a disjunction, one could have possibly more than two, but finitely many, exposed faces coming from a disjunction. The entire proof adapts easily to this minor change.\end{proof}

\begin{remark} The proofs of Theorems~\ref{thm:CP<=BB} and~\ref{thm:CP<=BB-II} imply the following. Let $C$ be a compact, convex set and let $F$ be an exposed face of $C$ defined by $\langle a, x \rangle = b$. Let $D$ be any valid disjunction (not necessarily facial) for some non-convexity $S$. Let $\langle c, x \rangle \leq \gamma$ be a cutting plane derived from $D$ for $F$. Then for any $\epsilon > 0$, there exists an $\epsilon$-approximate rotation $\langle c' , x \rangle \leq \gamma'$ of $\langle c , x \rangle \leq \gamma$ with respect to $\langle a, x \rangle = b$ such that $\langle c' , x \rangle \leq \gamma'$ is a cutting plane for $C$ derived from $D$.
\end{remark}

It is an open question to decide whether Theorem~\ref{thm:CP<=BB-II} extends to the case where $C$ is not required to be compact.
\medskip

Let us turn our attention to the independent set examples in Theorem~\ref{thm:CP<<BB} which show that cutting planes can be potentially much better in the 0/1 setting. These may seem a bit pathological because the disconnected components lead to a cartesian product of congruent polytopes, which is well-known to be a bad case for branch-and-bound. However, one can create other examples of the independent set problem which have exponential BB proofs of optimality that generalize these examples to connected graphs. We give two such constructions below, but many others are possible.

\begin{theorem}\label{thm:CP<<BB-II}
Given $m$ disjoint copies of $K_3$ (cliques of size $3$) and an extra vertex which is called the {\em center}, we arbitrarily add connections between the center and the other vertices to form a graph $G$ so that there is at most one vertex from each clique connected to the center. Then for $C=P(G)$ with objective $\sum_{v\in V}x_v$, $S= \Z^{3m+1}$ and $\D$ representing the family of all variable disjunctions, there is a cutting plane algorithm which solves the maximum independent set problem in at most $m$ iterations, but any branch-and-bound proof certifying an upper bound of $m+1$ on the optimal value has size at least $2^{m+1} - 2$ for all $m\geq 1$.
\end{theorem}

\begin{proof} Observe that if the center is never branched on, then one can simply use the same argument as in the proof of Theorem~\ref{thm:CP<<BB}. Now suppose we have a BB tree where at some nodes, the center is branched upon. We can replace the entire subtree at all such nodes with the subtree corresponding to the branch where the center node is set to $0$. This is a smaller BB tree with no branchings on the center node. We can now again apply the proof of Theorem~\ref{thm:CP<<BB}.

The CP algorithm is the same as in Theorem~\ref{thm:CP<<BB}.
\end{proof}

The following theorem adds edges between the copies of the cliques themselves to make the graph connected.

\begin{theorem}\label{thm:CP<<BB-III} Given $m$ disjoint copies of $K_m$ (cliques of size $m$) and a constant $\alpha$, where $0\leq \alpha<1$, we arbitrarily add connections between cliques to form a graph $G$ so that at most $\alpha m$ nodes of each clique are connected to other cliques. Then for $C=P(G)$  with objective $\sum_{v\in V}x_v$, $S= \Z^{m^2}$ and $\D$ representing the family of all variable disjunctions, any branch-and-bound proof certifying an upper bound of $m$ on the optimal value has size at least $2^{m+1}-2$ for all $m\geq 3/(1-\alpha)$. {In contrast, there is a CP proof based on variable disjunctions with $O(m^4)$ length.} \end{theorem}

\begin{proof} Consider first any feasible node of the branch-and-bound tree at depth $k < m$. Consider the path from the root node to this node at depth $k$ and suppose that of $k$ nodes on this path that were branched on, $k_1$ were set to $1$ and $k_2$ were set to $0$ (so $k_1 + k_2 = k$). Feasibility implies that all the $k_1$ vertices set to $1$ are from different cliques. From the remaining $m - k_1$ cliques, $k_2$ vertices were set to $0$. In addition, due to the connections between cliques, the vertices from these $m-k_1$ cliques connected to the $k_1$ vertices can take value at most $0$ in the LP at this node. There can be at most $\alpha m(m-k_1)$ such vertices, so the LP value is at least $k_1 + \frac{(m-k_1)m - k_2}{2} - \frac{\alpha m(m-k_1)}{2} = \frac{(1-\alpha)m^2 + (3-m+\alpha m)k_1 - k}{2}$, which is a nonincreasing function of $k_1$ and $k$, when $m\geq 3/(1-\alpha)$. Using the fact that $k_1 \leq k \leq m-1$, this lower bound on the LP value is at least $m+\frac{(1-\alpha)m-2}{2} > m$ for $m\geq 3/(1-\alpha)$. Thus, if the branch-and-bound tree has only feasible nodes, only nodes at depth $m$ or more can have LP values at most $m$. Since every branching node has at least two children, this means we have at least $2^{m+1}-1$ nodes.

To deal with BB trees with infeasible nodes, we proceed as in the proof of Theorem~\ref{thm:CP<<BB}.

{The CP reasoning is as follows. Define $P_0$ to be the fractional stable set polytope $P(K_m)$, and define recursively $P_ i = \conv((P_{i-1} \cap \{x \in \R^m: x_i = 0\}) \cup (P_{i-1} \cap \{x \in \R^m: x_i = 1\})$, for $i=1, \ldots, m$. As $P_{0} \cap \{x \in \R^m: x_1 = 0\}$ is described by the constraints $x_1 = 0,\;x_u + x_v  \leq  1,\; \forall u,v>1$ and $P_{0} \cap \{x \in \R^m: x_1 = 1\}$ is the point $x_1=1, x_u=0, \forall u>1$, then $P_1$ is described by the inequalities $x_1+x_u+x_v\le 1, \; \forall u,v>1$ (i.e., clique inequalities associated to all the triangles containing 1). More generally, it can be shown that $P_i$ is described by the system:

\[
\begin{array}{rcll}
(\sum_{j=1}^i x_j) + x_u + x_v & \leq & 1  &   \quad\forall u, v > i \\

0 \leq x \leq 1 &&&
\end{array}
\]
Therefore, $P_{m-2}$ is the integer hull of $P(K_m)$ and each $P_i$ has an $O(m^2)$ sized description. The CP proof simply proceeds through each copy of $P(K_m)$, convexifying sequentially to produce all the inequalities needed to describe $P_0, P_1, \ldots, P_{m-2}$. Since each $P_i$ has an $O(m^2)$ description, and we do $O(m)$ steps of convexification, all the facets of the integer hull of $P(K_m)$ can be derived with an $O(m^3)$ length proof. We repeat this for the $m$ copies of $K_m$, giving a total length of $O(m^4)$.}\end{proof}

{Note that in Theorems~\ref{thm:CP<<BB} and~\ref{thm:CP<<BB-II}, we provide CP algorithms that solve the problems in polynomially many iterations; whereas, for Theorem~\ref{thm:CP<<BB-III}, we are only able to obtain a CP proof and the question of obtaining a CP algorithm for these instances that takes $\poly(m)$ iterations remains open.}

\section{Future avenues of research}

Some open questions are highlighted in Table \ref{tab:CP-BB-results}. Here we elaborate on this. When the dimension is variable, although our results clarify the relative strength of BB and CP proofs/algorithms for the case of variable disjunctions, the situation is not clear if general split disjunctions are allowed. It can be shown that the independent set instances of Theorem \ref{thm:CP<<BB} can be solved in a polynomial number of iterations with BB based on general split disjunctions. This is because they can be solved in polynomial time by Chv\'atal-Gomory cuts (in fact, the cuts used in Claim~\ref{claim:CP-fast} are Chv\'atal-Gomory cuts) and a CP proof based on Chv\'atal-Gomory cuts can be viewed as a BB proof based on general split disjunctions where one side of every branching is empty. See also~\cite{beame_et_al:LIPIcs:2018:8341} where Chva\'tal-Gomory cuts are simulated by BB based on general splits. Thus, these independent set instances cannot be used to show that cutting planes are much better than branch-and-bound. {However, there are other known instances constructed by Dobbs et al. where BB proofs based on general split disjunctions are necessarily exponential in size; see Theorem 4.11 in~\cite{dash2014lattice}. These could potentially provide similar gaps between CP and BB based on general split disjunctions, but it is unknown if there are polynomial sized CP proofs based on general split disjunctions for these instances.}  It is also possible that there are 0/1 instances where BB based on general split disjunctions can be shown to perform much better than CP, as there exist 0/1 polytopes for which any CP proof of infeasibility based on general split disjunctions has exponential length \cite{dash2010complexity}. Finally, for general polytopes, it is not clear whether there are examples where BB is much better than CP with general split disjunctions, as is the case for variable disjunctions (Theorems \ref{thm:BB<<CP} and \ref{thm:BB-con-CP-exp}).

When the dimension is fixed, a question that remains open is whether there are examples in which BB based on general split disjunctions takes a constant number of iterations while CP based on the same disjunctions requires a number of iterations that is exponential in the input size. (Note that this is the case when only variable disjunctions are allowed: see Theorem \ref{thm:BB-con-CP-exp}.) {In fact, nothing rules out the possibility of a CP algorithm based on split disjunctions that takes polynomial number of iterations in fixed dimensions.} This would complement Lenstra style algorithms for integer optimization in fixed dimensions. Such a CP algorithm has been found for 2 dimensions \cite{basu-split-2D}.

We also suspect that a related conjecture is true: the split rank of any rational polyhedron is bounded by a polynomial function of the input size, even for variable dimension. A fixed dimension CP algorithm would prove this latter conjecture for fixed dimensions. {It is well-known that the split rank of any two dimensional polyhedron is at most two; for instance, see~\cite{basu-split-2D}. To the best of our knowledge, it is not known if the conjecture is true even for three dimensional rational polyhedra.}

Another open question is related to Theorem \ref{thm:CP<=BB}, which asserts that cutting planes are always at least as good as branch-and-bound up to $\epsilon$-slack. Since, for any $\epsilon>0$, the CP proof has the same length as the BC proof, it is natural to wonder whether one can get the same result with $\epsilon=0$.

An intriguing open question comes out of the discussion at the very end of Section~\ref{sec:intro}. Can we make the difference between a CP/BB algorithm and CP/BB proof sharper, extending the insight from~\cite{owen2001disjunctive}? In particular, in the example from~\cite{owen2001disjunctive} CP algorithms do not converge to the optimal value, but there are finite length CP proofs of arbitrary accuracy; however, there are no finite cutting plane proofs of optimality making the example somewhat contrived. We pose a sharper version of the problem below.

Let $\D$ be a family of disjunctions such that there are finite length CP proofs of optimality (or infeasibility) for every instance. Is it true that for every instance in dimension $n \in \N$, the ratio of the smallest number of iterations of a CP algorithm and the smallest length of a CP proof is bounded by $\poly(n)$, or is there a family of instances such that this ratio is lower bounded by $\Omega(\exp(n))$? The same question goes for BB algorithms and proofs.

An answer to this question would give insight into the wisdom of the widespread practice of using only cutting planes that work on the optimal solution.

\subsection*{Acknowledgments}
Michele Conforti and Marco Di Summa were supported by the PRIN grant 2015B5F27W and by a SID grant of the University of Padova. Amitabh Basu and Hongyi Jiang gratefully acknowledge support from NSF Grant CMMI1452820 and ONR Grant N000141812096. {Daniel Dadush pointed out to us the example by Dobbs et al. from~\cite{dash2014lattice} that establishes an exponential lower bound on any BB proof based on general split disjunctions.} The manuscript also benefited greatly from discussions with Aleksandr Kazachkov, Andrea Lodi and Sriram Sankaranarayanan. Part of this work was done when the first author was visiting the Centre de Recherches Math\'ematiques (CRM) at Universit\'e de Montr\'eal, Canada as part of the Simons-CRM scholar-in-residence program. The visit and the research done was supported in part by funding from the Simons Foundation and the Centre de Recherches Math\'ematiques.

\bibliographystyle{plain}
\bibliography{../../full-bib}
\end{document}